\documentclass{article}
\usepackage[utf8]{inputenc}
\usepackage{multicol}
\usepackage{caption} 
\usepackage{subcaption}
\usepackage{tabularx}
\usepackage{capt-of}
\usepackage{tikz}
\usepackage{float}
\usepackage{lineno}
\usepackage{graphicx}
\usepackage{comment}
\usepackage[shortlabels]{enumitem}
\usepackage{amsthm,amsmath, amsfonts}
\usepackage{fancyhdr} 
\usepackage{appendix} 
\usepackage{caption} 
\usepackage{hyperref}
\usepackage{placeins}
\usetikzlibrary{calc}
\usepackage{lineno}

\newtheorem{theorem}{Theorem}
\newtheorem*{theorem*}{Theorem}
\newtheorem{proposition}[theorem]{Proposition}
\newtheorem{corollary}[theorem]{Corollary}
\newtheorem{lemma}{Lemma}
\newtheorem{definition}{Definition}

\newtheorem{question}{Question}
\newtheorem{example}{Example}

\newcommand{\arck}[1]{{\overset{\frown}{#1}}}

\DeclareMathOperator{\ch}{ch}

\DeclareMathOperator{\CH}{CH}

\PassOptionsToPackage{hyphens}{url}

\title{Convex geometries representable with colors, by ellipses on the plane, and impossible by circles}
\author{Kira Adaricheva \thanks{Department of Mathematics, Hofstra University,
Hempstead, NY 11549, USA (kira.adaricheva@hofstra.edu)}
\and Evan Daisy 
\thanks{Amherst College, USA}
\and Ayush Garg \thanks{Indian Institute of Technology, Delhi, India}
\and Zachary King 
\and Grace Ma \thanks{University of Notre Dame, South Bend, USA}
\and Michelle Olson \thanks{California State University, Fullerton, USA}
\and Cat Raanes \thanks{Carnegie Mellon University, Pittsburgh, USA}
\and James Thompson \thanks{University of North Carolina, Chapell Hill, USA}
}
\date{August 2020 - May 2022}
\begin{document}
\maketitle


\begin{abstract} A convex geometry is a closure system satisfying the anti-exchange property. This paper, following the work of K. Adaricheva and M. Bolat (2016) and the Polymath REU 2020 team, continues to investigate representations of convex geometries on a 5-element base set. It introduces several properties: the opposite property, nested triangle property, area Q property, and separation property, of convex geometries of circles on a plane, preventing this representation for numerous convex geometries on a 5-element base set. It also demonstrates that all 672 convex geometries on a 5-element base set have a representation by ellipses, as given in the appendix for those without a known representation by circles, and introduces a method of expanding representation with circles by defining unary predicates, shown as colors. 

\end{abstract}

\tableofcontents

\section{Introduction}

This paper follows the first paper of the convex geometries team \cite{Poly20} developed during the Polymath REU - 2020 project, see more details in \cite{Lemons21}. 

In both papers we addressed the problem raised in Adaricheva and Bolat \cite{AdBo19}: whether all convex geometries 
with convex dimension at most 5 are representable by circles on the plane using the closure operator of convex hull for circles. 

A convex geometry is a closure system $(X, \phi)$, where closure operator $\phi$ satisfies the anti-exchange property. The survey paper of Edelman and Jamison  \cite{EdJa85} was instrumental to start off the study of finite convex geometries, which are also the dual systems to antimatroids. The study of infinite convex geometries was initiated in Adaricheva, Gorbunov and Tumanov \cite{AGT03}. To see the development of the topic, including infinite convex geometries, one needs to consult the more recent survey Adaricheva and Nation \cite{AdNa16}. 

In \cite{Cz14} G. Cz\'edli proposed representation of convex geometries by interpreting elements of base set $X$ as circles on the plane, and closure operator $\phi$ as a convex hull operator acting on circles. He also showed that all finite convex geometries with \emph{convex dimension} at most 2 can be represented by circles on the plane. The parameter $cdim$ of convex geometry represents the diversity of closed sets with respect to closure operator $\phi$ by measuring the size of maximal anti-chain of closed sets, see Edelman and Saks \cite{EdSa88}. In Adaricheva and Bolat \cite{AdBo19}, it was found that there is an obstruction for representation of convex geometries by circles on the plane, in the form of the Weak Carousel property, which allowed the authors to build an example of a convex geometry on a 5-element set of $cdim=6$.

In \cite{Poly20} it was discovered that all convex geometries for $|X|=4$ are representable by circles on the plane, and among 672 non-isomorphic geometries for $|X|=5$, 621 were represented, while 49 were deemed \textbf{``impossible"}, i.e. no circle representation was found for them. Finally, it was proved in \cite{Poly20} that 7 of 49 impossible geometries cannot be represented due to the Triangle property, in addition to one from \cite{AdBo19} due to the Weak Carousel property.
In particular, some geometries are non-representable due to the Triangle property had $cdim=4,5$.

In the current paper, we prove four new geometric properties related to circles on the plane that allow us to prove an additional 30 of the 49 geometries from the impossible list given in \cite{Poly20} are not representable.

We identify these geometries by their number in the list as generated in \cite{Poly20}, and additionally by their tight implications, sometimes with the elements of the set relabeled for clarity.

1) Opposite Property (Theorem \ref{OppProp} and Corollary \ref{List1}) implies that the following geometries are not representable:
\[
G46, G60, G74, G84, G114, G115, G153
\]

2) Nested Triangles Property (Theorem \ref{NT} and Corollary \ref{List2}):
\begin{align*}
&G12,G15,G18,G21,G23,G26,G27,G33,G35,\\
&G45,G47,G49,G56,G60,G70,G89,G94,G134
\end{align*}

3) Area Q Property (Theorem \ref{G74} and Corollary \ref{List3}): 
\[
G74,G96,G105,G143,G147,G206,G235,G351
\]

4) Separation Property (Theorem \ref{G14} and Corollary \ref{List4}):
\[
G14, G23, G26, G39, G54.
\]

Note that these three lists slightly overlap, with several of the geometries appearing in more than one list, and a few appeared earlier due to the Triangle Property in \cite{Poly20}.

This provides a partial answer to Question 2 in \cite{Poly20}.

The following 11 impossible geometries comprise still unproven cases:
\[
G43,G57,G69,G87,G95,G122,G129,G132,G161,G175,G211
\]

Through the work on these ``impossible" geometries, we used a program, previously developed for \cite{Poly20}, which searched for all convex geometries that satisfied a given list of implications. In addition, we also used the GeoGebra software \cite{ggb}, to help with proofs and various representations.



Finally, we address two alternate schemes of representation.

In section \ref{ellipse} we discuss representation of geometries by ellipses on the plane. It was shown in J. Kincses \cite{Kin17} that not all geometries can be represented by ellipses. However, we demonstrate in Appendix A that all 49 geometries on a 5-element set that were deemed ``impossible" for circle representation in \cite{Poly20}, in fact, have representation by ellipses. 

\begin{question}
What is the smallest cardinality of the base set and the smallest convex dimension of a geometry that cannot be represented by ellipses on the plane?
\end{question}

In the last section \ref{predicates} we introduce a new type of representation of geometries using unary predicates on the base set that can be interpreted as colors. In Appendix B we give representation of all 49 ``impossible" convex geometries on a 5-element set using colored circles. It is also shown in section \ref{predicates} that one of the geometries, G18, cannot be represented with two colors only. The rest of the ``impossible" geometries are represented with 2 colors only.
Note that G18 has $cdim=6$, thus, we may ask: 

\begin{question}
Can all convex geometries of convex dimension 4 or 5 be represented by circles with 2 colors only?
\end{question}

\section{Terminology and known results}

A convex geometry is a special case of a closure system. It can be defined through a closure operator, or through an alignment.
\begin{definition}\label{def:closure}
Let $X$ be a set. A mapping $\varphi \colon 2^X \to 2^X$ is called a \emph{closure operator}, if for all $Y, Z \in 2^X$:
\begin{enumerate}[noitemsep]
    \item $Y \subseteq \varphi(Y)$
    \item if $Y \subseteq Z$ then $\varphi(Y) \subseteq \varphi(Z)$
    \item $\varphi(\varphi(Y)) = \varphi(Y)$
\end{enumerate}
 
A subset $Y \subseteq X$ is \emph{closed} if $\varphi(Y) = Y$. The pair $(X,\varphi)$, where $\varphi$ is a closure operator, is called a \emph{closure system}.
\end{definition}

\begin{definition}\label{def:alignment}
Given any (finite) set $X$, an \textit{alignment} on $X$ is a family
$\mathcal{F}$ of subsets of $X$ which satisfies two properties:
\begin{enumerate} [noitemsep]
    \item $X \in \mathcal{F}$
    \item If $Y, Z \in \mathcal{F}$ then $Y \cap Z \in \mathcal{F}$.
\end{enumerate}
\end{definition}
\noindent{Closure systems are dual to alignments in the following sense:}
\medskip

Often, a closure operator is described by the means of \emph{implications}. For example $Y\rightarrow z$ means that $z \in \varphi(Y)$, for the underlying closure operator $\varphi$. Note that every closure operator can be expressed by some set of implications, for example, $\Sigma_\varphi=\{A\rightarrow \varphi(A): A \in 2^X\}$. On the other hand, given \emph{any} set of implications $\{A\rightarrow B: A,B\subseteq X, B\not = \emptyset\}$, one can find a unique minimal closure operator $\varphi$ such that $B\subseteq \varphi(A)$.
\medskip

\noindent If $(X, \varphi)$ is a closure system, one can define a family of closed sets $\mathcal{F} :=\{Y \subseteq X : \varphi(Y) = Y\}$. Then $\mathcal{F}$ is an alignment.
\medskip

\noindent If $\mathcal{F}$ is an alignment, then define $\varphi: 2^X \rightarrow 2^X$ in the
following manner:

\medskip

\noindent for all $Y \subseteq X$, let 
$\varphi(Y):=\bigcap  \{Z \in \mathcal{F}
: Y \subseteq Z\}$.
Then $(X, \varphi)$ is a closure system.

\noindent A useful perspective on closure systems is provided by its \textit{implications.} Given a closure system $(X, \varphi)$, 
an ordered pair $(A,B)$, where $A, B \in 2^X\setminus \{\emptyset\}$, also denoted 
$A \to B$, is an \emph{implication} if $B \subseteq \varphi(A)$.

\noindent 

A set of implications fully defining closure operator $\varphi$ is often referred to as \emph{implicational basis} of a closure system. The recent exposition on implicational bases is given in Adaricheva and Nation \cite{AN16I}.

\begin{definition}\label{def:cg}
A closure system $(X,\varphi)$ is called a \emph{convex geometry} if
\begin{enumerate}[noitemsep]
    \item $\varphi(\emptyset) = \emptyset$
    \item For any closed set $Y\subseteq X$ and any distinct points $x,y \in X\setminus Y,$ if $x \in
    \varphi(Y \cup \{y\})$ then $y \not\in \varphi(Y \cup \{x\})$.
\end{enumerate}
\end{definition}
The second property in this definition is called the \emph{Anti-Exchange Property}.

\noindent We can use duality between closure operators and alignments to provide another definition of a convex
geometry.
\begin{definition}
\label{cg_alignment}
A closure system $(X,\varphi)$ is a convex geometry iff the corresponding alignment
$\mathcal{F}$ satisfies the following two properties:
\begin{enumerate}[noitemsep]
    \item $\emptyset \in \mathcal{F}$
    \item If $Y \in \mathcal{F}$ and $Y \neq X$ then $\exists a\in X\setminus Y$
    s.t. $Y\cup\{a\} \in \mathcal{F}$.
\end{enumerate}
\end{definition}

A particular example of a closure operator on a set is the convex hull operator, where the base set $X$ is a set of points in Euclidean space $\mathbb{R}^n$. For the goals of this paper we use only the plane, i.e. the space $\mathbb{R}^2$.

\begin{definition}
A set $S$ in $\mathbb{R}^2$ is called  \emph{convex} if for any two points $p,q \in S$, the line segment connecting $p$ and $q$ is also contained in $S$. 
\end{definition}

\begin{definition}
Given set $S$ of points in $\mathbb R ^ 2$, the convex hull of $S$, in notation $\CH(S)$, is the intersection of all convex sets in $\mathbb R ^2$ which contain $S$.
\end{definition}

Thus, $\CH$ can be thought as a closure operator acting on $\mathbf{R}^2$.

Finally, we want to recall the definition of the convex hull operator for circles introduced in Cz\'edli \cite{Cz14}.

If $x$ is a circle on the plane, then by $\tilde{x}$ we denote a set of points belonging to $x$. We allow a circle to have a radius $0$, in which case it is a point.

\begin{definition}
Let $X$ be a finite set of circles in $\mathbb{R}^2$. Define the operator $\ch_c : 2^X \rightarrow 2^X$, a convex hull operator for circles, as follows:
\[
\ch_c(Y) =\{x \in X: \tilde{x}\subseteq \CH(\bigcup_{y \in Y}\tilde{y})\},
\]
for any $Y\in 2^X$.
\end{definition}
For example, $a, e \in \ch_c(b,c,d)$ on Figure \ref{fig:ch1}, while $a \not \in \ch_c(b,c,d)$, $e \in \ch_c(b,c,d)$ in Figure \ref{fig:ch2}.

We think of any circle $d$ as a \emph{closed} set of points (in standard topology of $\mathbb{R}^2$), and we use $Circ(d)$ for the boundary of $d$, which is included in $d$. Often we will be using the same notation $d$ for an element of convex geometry of circles, but also as the set of points in $\mathbb{R}^2$, thus avoiding notation $\tilde{d}$. For example, $abc\rightarrow d$ means that $d \in \ch_c(a,b,c)$ in convex geometry of circles, but $c\subseteq d$ should be understood as $\tilde{c}\subseteq \tilde{d}$.

\begin{figure}
\begin{center}
    \begin{subfigure}[a]{0.4\textwidth}
        \centering
      \includegraphics[width=\textwidth]{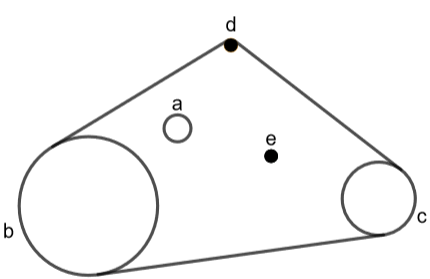}
      \caption{Circle $a$ is in the convex hull of $b,c,d$}
           \label{fig:ch1}
    \end{subfigure}
\hspace{1cm}
    \begin{subfigure}[a]{0.4\textwidth}
        \centering
        \includegraphics[width=\textwidth]{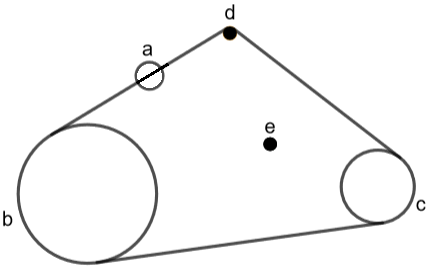}
        \caption{Circle $a$ is not in the convex hull of $b,c,d$}
            \label{fig:ch2}
    \end{subfigure}
\end{center}
\caption{Examples of convex hulls}
\label{fig:convex hull example 2}
\end{figure}

We often will be writing $\CH(Y)$, for the subset $Y$ of circles, instead of $\CH(\bigcup_{y \in Y}\tilde{y})\}$. For example, $\CH(a,b,c)$ for circles $a,b,c\in X$ represents the convex set of points $\CH(\tilde{a}\cup\tilde{b}\cup \tilde{c})$.

The following statement is modified Lemma 5.1 in \cite{AdBo19}.

\begin{lemma}\label{3 circles}
Let $a,b,c$ be three circles on the plane such that none is included into another. 
Then there are three different types of configurations formed by  $a,b,c$:

\begin{itemize}
    \item[(1)] $\CH(x,y,z)= \CH(y,z)$, for some labeling of $a,b,c$ by $x,y,z$ as on Figure \ref{fig:WeakCarousel1};
    \item[(2)] $\CH(x,y,z)=\CH(x,y)\cup \CH(y,z)$, for some labeling of $a,b,c$ by $x,y,z$, as on Figure \ref{fig:WeakCarousel2} or Figure \ref{fig:WeakCarousel2Case}.
    \item[(3)] $\CH(x,y,z)\not = \CH(x,y)\cup \CH (x,z)$, for any labeling of $a,b,c$ by $x,y,z$, as on Figure \ref{fig:WeakCarousel3} or 
    Figure \ref{fig:WeakCarousel4}.
\end{itemize}
\end{lemma}

\begin{figure}[H]
    \centering
    \begin{tikzpicture}[scale=.4]
	\draw (-0.5,1) circle (1);
	\draw  (2, 1.5) circle(1.5);
	\draw (5,1.25) circle (0.75);
	\draw[-] (-0.5,0) -- (1.95,0);
	\draw[-] (-0.68,1.985) -- (1.6,2.95);
	\draw[-] (2.35,2.96) -- (5.12, 2); 
	\draw[-] (2.2,0) -- (5.12, 0.5); 
	\node at (-0.5,0.95) {$a$};
	\node at (2,1.45) {$b$};
	\node at (5, 1.2) {$c$};
\end{tikzpicture} 
    \caption{}
    \label{fig:WeakCarousel2}
\end{figure}
\begin{center}
\captionsetup{type=figure}
\begin{tabularx}{\textwidth} { 
   >{\centering\arraybackslash}X 
   >{\centering\arraybackslash}X 
   >{\centering\arraybackslash}X | }

 \begin{tikzpicture}[scale=.4]

	\draw (-0.5,1) circle (1); 
	\draw  (1.4, 1.2) circle(0.75); 
	\draw (3.9,1.25) circle (1.25); 

	\draw[-] (-0.5,0) -- (4,0); 
	\draw[-] (-0.68,1.985) -- (3.95,2.5);
	\node at (-0.5,0.95) {$a$};
	\node at (1.46,1.25) {$b$};
	\node at (4.0,1.25) {$c$};
\end{tikzpicture}\captionof{figure}{}\label{fig:WeakCarousel1} 
 & 
 \begin{tikzpicture}[scale=.4]
	\draw (-0.5,1) circle (1);
	\draw  (2, 1.5) circle(1.5);
	\draw (5,1.25) circle (1.25);
	\draw[-] (-0.5,0) -- (5,0);
	\draw[-] (-0.68,1.985) -- (1.6,2.95);
	\draw[-] (2.35,2.96) -- (5.12, 2.49); 
	\node at (-0.5,0.95) {$a$};
	\node at (2,1.45) {$b$};
	\node at (5.1,1.2) {$c$};
      \end{tikzpicture} \captionof{figure}{}\label{fig:WeakCarousel2Case} \\
\end{tabularx}
\end{center}

\begin{center}
\captionsetup{type=figure}
\begin{tabularx}{\textwidth} { 
   >{\centering\arraybackslash}X 
   >{\centering\arraybackslash}X 
   >{\centering\arraybackslash}X | }
 \begin{tikzpicture}[scale=.4]
	\draw (0.35,-0.65) circle (1);
	\draw  (2, 1.5) circle(1.5);
	\draw (4,-1) circle (1.25);
	\draw[-] (0.3,-1.65) -- (4,-2.25);
	\draw[-] (-0.65,-.5) -- (0.73,2.35);
	\draw[-] (3.25,2.35) -- (4.9, -0.1); 
	\node at (0.3,-0.65) {$a$};
	\node at (2,1.45) {$b$};
	\node at (4,-1) {$c$};
       \end{tikzpicture}\captionof{figure}{}\label{fig:WeakCarousel3} 
 & 
 \begin{tikzpicture}[scale=.4]
	\draw (0,-1) circle (0.75);
	\draw  (0.5, 0.75) circle(1);
	\draw (4,0) circle (2.5);
	\draw[-] (-0.75,-1.1) -- (-0.47,1.05);
	\draw[-] (0.2,1.7) -- (3.25, 2.35);
	\draw[-] (-0.1,-1.77) -- (3.5,-2.45);
	\node at (0,-1) {$a$};
	\node at (0.45, 0.73) {$b$};
	\node at (4,0) {$c$};
       \end{tikzpicture}\captionof{figure}{}\label{fig:WeakCarousel4}    \\

\end{tabularx}
\end{center}

\begin{definition}
Let $x, y, z$ be circles in configuration 1 or 2. We say that circle $y$ is in the \emph{center} (of configuration), if $\CH(x, y, z) = \CH(x,y) \cup \CH(y,z)$ or  $y \in ch_c(x, z)$. We call circles $x$ and $z$ \emph{external circles}.
\end{definition}

\begin{definition}
Let $x,y,z$ be in configuration 2 such that $y$ has a touching point on only one tangent of $x,z$ and 2 points of intersection on the other tangent; then $x,y,z$ is in \emph{limit case 2}. 
\end{definition}

\begin{lemma}\label{Lemma42}
Suppose $p$ is a circle inscribed in $\angle{A}$ of a triangle $\triangle{ABC}$. Let $w_1$ be an area of $\triangle{ABC}$ outside the circle $p$ and adjacent to the vertex $A$, and $w_2= \triangle{ABC}\setminus (p \cup w_1)$ as in Figure \ref{fig:PinAngle}. Then any circle $y\subseteq \triangle{ABC}$ cannot have non-empty intersections with both areas $w_1, w_2$.
 \end{lemma}
\begin{figure}[H]
\centering
\vspace{1.5mm}
\includegraphics[width=0.27\textwidth]{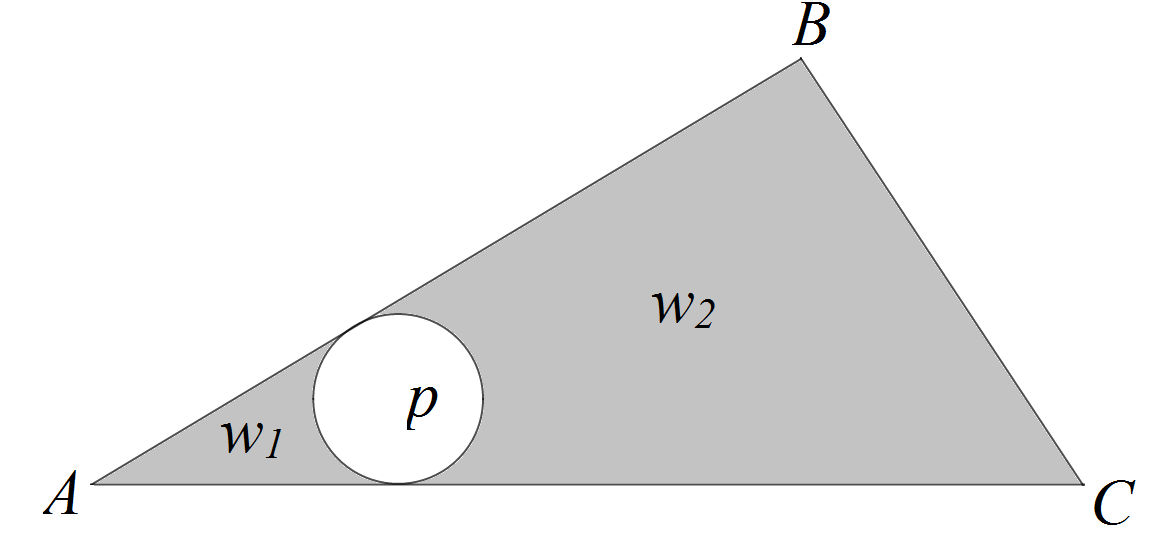}  
\caption{ }
  \label{fig:PinAngle}
\end{figure}

Note that the version of this statement is still true and also addressed in \cite{AdBo19}, when circle $p$ touches two parallel lines $(AB)$ and $(A'C)$, and $w_1$, $w_2$ denote two disjoint parts of a strip between these lines.

The next result was a crucial observation about tight implication for circles from \cite{Poly20}. The underlying closure operator here is $\ch_c$.\\

\begin{definition}
Let $(X,\varphi)$ be a closure system. An implication $Y\to u$, $Y\subseteq X, u\in X$, is called \emph{tight}, if implications $(Y\setminus \{z\}) \to u$ do not hold, for all $z \in Y$. We will call implication $Y\to u$ \emph{loose}, if $u\in \varphi(Y)$ and it may not be tight.
\end{definition}

\begin{lemma}[The Triangle Property]\label{triangle}
If $a,b,c,$ and $e$ are circles in the plane with centers $A,B,C,$ and $E$ respectively, and $abc\to e$ is a tight implication, then $E$ must lie in the interior of triangle formed by $A$,$B$ and $C$.
\end{lemma}

\begin{lemma}\label{OpAngle}
Given points $A,B,C,E$ in the plane, for $E$ to lie in the interior of the triangle formed by $A$,$B$ and $C$, $C$ must be located in the opposite angle to $\angle AEB$.
\end{lemma}
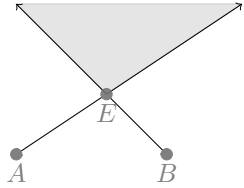
\begin{figure}[H]
    \centering
        \begin{tikzpicture}
            \draw[<->] (-1,-1) -- (2,1);
            \filldraw [gray] (-1,-1) circle[radius=0.075] node[anchor=north] {$A$};
            \draw[<->] (1,-1) -- (-1,1);
            \filldraw [gray] (1,-1) circle[radius=0.075] node[anchor=north] {$B$};
            \filldraw [gray] (1/5,-1/5) circle[radius=0.075] node[anchor=north] {$E$};
            \draw[opacity=0.2,fill=gray] (2,1) -- (1/5,-1/5) -- (-1,1) -- (2,1);
        
        \end{tikzpicture}
    \caption{Possible locations for $C$ shaded}
    \label{fig:1A}
\end{figure}

We add one new observation we will need later  in our paper.

\begin{lemma}\label{Overlap}
Let circles $a,b,c$ be in configuration (3) of Lemma \ref{3 circles}. If $a,c$ overlap, then $\CH(a,b,c) = \CH(b,c)\cup \CH (a,b)\cup \CH(a,c)$.
\end{lemma}
\begin{proof}
Let $A,B,C$ be centers of circles $a,b,c$, respectively. Every point in $\CH(a,b,c)$, which is not in $\triangle ABC$, will belong to one of $\CH(b,c)$,$\CH (a,b)$, $\CH(a,c)$. Therefore, we only need to consider a point $D \in \triangle ABC$. Since $a,c$ overlap, segment $[AC]$ is covered by a union of two radii segments on the line $(AC)$: radius $[AA_1]$ of circle $a$ and radius $[CC_1]$ of circle $c$.
Let $(BD)$ intersect $[AC]$ at point $P$. Then either $P\in [AA_1]\subseteq \tilde{a}$, or $P\in [CC_1]\subseteq \tilde{c}$. Since $D \in [BP]$, we have $D \in \CH(b,a)$ or $D\in \CH(b,c)$.
\end{proof}

\section{Opposite Property}

In this section we show that circle configurations on the plane satisfy the property that we call the Opposite Property. It shows that 7 geometries that were included in the list of \emph{impossible geometries} in paper \cite{Poly20}  cannot be represented by circles on the plane.

\medskip
\begin{theorem}\label{OppProp}
If $bcd\to e$ is a tight implication and $ab\to e$, $ac\to e$, and $ad\to e$, then $a\to e$.
\end{theorem}

\begin{proof}
Suppose the hypotheses hold, and let $A,B,C,D,E$ be the centers of circles $a,b,c,d,e$ respectively. If $A=E$  then either $a\to e$, in which case the theorem holds, or $e\to a$, in which case $ab\to e$ is only possible if $b\to e$, but then $bcd\to e$ is not tight. If $A \neq E$, we have the following picture, where the distance between $E$ and $E'$ is the radius of $e$ (i.e. $E'$ is the point in $e$ farthest away from $A$; note that $E$ and $E'$, and thus $\ell$ and $\ell'$, are not necessarily distinct):

\begin{figure}[H]
    \centering
        \begin{tikzpicture}
            \draw[thick, <->] (0,0) -- (4.5,0);
            \filldraw [gray] (1,0) circle[radius=0.075] node[anchor=south] {$A$};
            \filldraw [gray] (3,0) circle[radius=0.075] node[anchor=south east] {$E$};
            \filldraw [gray] (4,0) circle[radius=0.075] node[anchor=south east] {$E'$};
            \draw[thick, <->] (3,-3) -- (3,3) node[anchor=south east] {$\ell$};
            \draw[<->] (4,-3) -- (4,3) node[anchor=south east] {$\ell'$};
            \draw (3,0.2) -- (3.2,0.2) -- (3.2,0);
            \draw (4,0.2) -- (4.2,0.2) -- (4.2,0);
        \end{tikzpicture}
    \label{fig:1B}
\end{figure}

$\ell'$ splits the plane into two semi-planes, and if any point in $a$ is in the semi-plane opposite $A$, then clearly $a\to e$ and the theorem holds. 

If $a$ is contained in the open semi-plane containing $A$, then in order for $ab\to e$, $ac\to e$, and $ad\to e$, $b,c$ and $d$ must all have at least one point in the semi-plane opposite $A$ (otherwise both circles will be contained in the same open semi-plane not containing $E'$, and thus their convex hull cannot contain $E'$). 

$\ell$ also splits the plane into two semi-planes and since, by the triangle property (lemma 5), $E$ is in the interior of the triangle formed by $B,C$ and $D$, at least one of $B,C,D$ must be contained in the same open semi-plane as $A$.
Without loss of generality, suppose this is true of $B$. $B$ cannot be on the ray from $E$ through $A$ (else $b\to e$, in which case $bcd\to e$ is not tight), so $B$ must be contained an open quadrant formed by $(AE)$ and $\ell$.

Call this quadrant ``quadrant I" (shaded below) and label the other quadrants II, III, and IV, according to the figure below and which semi-planes formed by $(AE)$ and $\ell$ they share with $B$. 

Since $b$ is a circle that must have at least one point in the semi-plane opposite $A$ with respect to $\ell'$, as noted above, $b$ must contain the point closest to $B$ in this semi-plane, namely the intersection of $\ell'$ and a line through $B$ parallel to $(AE)$, which we will call $B'$. Thus $b$ must at least contain the circle of minimal radius centered at $B$ that includes $B'$. 

Call this circle, pictured below with a dashed line, $b'$.

\begin{figure}[H]
    \centering
        \begin{tikzpicture}
            \filldraw [gray] (0,3) circle[radius=0.001] node[anchor=north west] {I};
            \filldraw [gray] (0,-3) circle[radius=0.001] node[anchor=south west] {II};
            \filldraw [gray] (4.8,-3) circle[radius=0.001] node[anchor=south east] {III};
            \filldraw [gray] (4.8,3) circle[radius=0.001] node[anchor=north east] {IV};
            \filldraw [draw opacity=0, fill opacity=0.05] (3,0) rectangle (0,3);
            \draw[thick, <->] (0,0) -- (4.5,0) node[anchor=west] {$(AE)$};
            \filldraw [gray] (2.5,2) circle[radius=0.075] node[anchor=north east] {$B$};
            \draw [dashed] (2.5,2) circle[radius=1.5];
            \filldraw [gray] (4,2) circle[radius=0.075] node[anchor=south west] {$B'$};
            \draw[<->] (0,2) -- (4.5,2);
            \draw (3,2.2) -- (3.2,2.2) -- (3.2,2.0);
            \filldraw [gray] (3,0) circle[radius=0.075] node[anchor=north east] {$E$};
            \draw (3,0) circle[radius=1];
            \filldraw [gray] (4,0) circle[radius=0.075] node[anchor=north west] {$E'$};
            \draw[thick, <->] (3,-3) -- (3,3) node[anchor=south] {$\ell$};
            \draw[<->] (4,-3) -- (4,3) node[anchor=south] {$\ell'$};
            \draw (3,0.2) -- (3.2,0.2) -- (3.2,0);
            \draw (4,0.2) -- (4.2,0.2) -- (4.2,0);
            \draw[<->] (3.75,-3) -- (2.25,3);
        \end{tikzpicture} \hfill
        \begin{tikzpicture}
            \filldraw [gray] (0,3) circle[radius=0.001] node[anchor=north west] {I};
            \filldraw [gray] (0,-3) circle[radius=0.001] node[anchor=south west] {II};
            \filldraw [gray] (4.8,-3) circle[radius=0.001] node[anchor=south east] {III};
            \filldraw [gray] (4.8,3) circle[radius=0.001] node[anchor=north east] {IV};
            \filldraw [draw opacity=0, fill opacity=0.05] (3,0) rectangle (0,3);
            \draw[thick, <->] (0,0) -- (4.5,0) node[anchor=west] {$(AE)$};
            \filldraw [gray] (2.8,0.5) circle[radius=0.075] node[anchor=north east] {$B$};
            \draw [dashed] (2.8,0.5) circle[radius=1.2];
            \filldraw [gray] (4,0.5) circle[radius=0.075] node[anchor=south west] {$B'$};
            \draw[<->] (0,0.5) -- (4.5,0.5);
            \draw (3,0.7) -- (3.2,0.7) -- (3.2,0.5);
            \filldraw [gray] (3,0) circle[radius=0.075] node[anchor=north east] {$E$};
            \draw (3,0) circle[radius=1];
            \filldraw [gray] (4,0) circle[radius=0.075] node[anchor=north west] {$E'$};
            \draw[thick, <->] (3,-3) -- (3,3) node[anchor=south] {$\ell$};
            \draw[<->] (4,-3) -- (4,3) node[anchor=south] {$\ell'$};
            \draw (3,0.2) -- (3.2,0.2) -- (3.2,0);
            \draw (4,0.2) -- (4.2,0.2) -- (4.2,0);
            \draw[<->] (4.2,-3) -- (1.8,3);
        \end{tikzpicture}
    \label{fig:2A}
\end{figure}

Now $(BE)$ splits the plane into two semi-planes, and since $E$ is in the interior of the triangle formed by $B,C$ and $D$, one of $C,D$ must be in the open semi-plane that does not contain $B'$, so without loss of generality suppose it is $C$. Note that, regardless of the placement of $B$, this semi-plane is contained within quadrants I, II, and III. As with $b$, the convex hull of $c$ also must contain the point closest to $C$  in the semi-plane opposite $A$ with respect to $\ell'$, namely the point where $\ell'$ and a line through $C$ parallel to $(AE)$ intersect, or $C$ itself if $C$ is already in this semi-plane. The convex hull of $c$ must at least contain the circle of minimal radius centered at $C$ that includes $C'$, which we will call $c'$, so the convex hull of $b'$ and $c'$ is contained in the convex hull of $b$ and $c$.

If $C$ is not in quadrant I, it must be contained in the open semi-plane opposite $B$ with respect to $(AE)$. First consider a placement of $C$ on the line $(BE)$ (although $C$ can in fact only be arbitrarily close to such a placement). In this case the convex hull of $b$ and $c$ is symmetric about $(BE)$. If $C$ is in the same semi-plane as $B$ with respect to $\ell'$, i.e. $C \neq C'$, then $B',E',$ and $C'$ are all on $\ell'$, so $b', e,$ and $c'$ all share a common tangent line. They must also share this common tangent reflected about $(BE)$, the line through their centers, so $e$ is in the convex hull of $b'$ and $c'$ and thus also of $b$ and $c$. 

\begin{figure}[H]
    \centering
        \begin{tikzpicture}
            \filldraw [gray] (0,3) circle[radius=0.001] node[anchor=north west] {I};
            \filldraw [gray] (0,-3) circle[radius=0.001] node[anchor=south west] {II};
            \filldraw [gray] (4.8,-3) circle[radius=0.001] node[anchor=south east] {III};
            \filldraw [gray] (4.8,3) circle[radius=0.001] node[anchor=north east] {IV};
            \draw[thick, <->] (0,0) -- (4.5,0) node[anchor=west] {$(AE)$};
            \filldraw [gray] (2.5,2) circle[radius=0.075] node[anchor=north east] {$B$};
            \draw [dashed] (2.5,2) circle[radius=1.5];
            \filldraw [gray] (4,2) circle[radius=0.075] node[anchor=south west] {$B'$};
            \draw[<->] (0,2) -- (4.5,2);
            \draw (3,2.2) -- (3.2,2.2) -- (3.2,2.0);
            \filldraw [gray] (3,0) circle[radius=0.075] node[anchor=north east] {$E$};
            \draw (3,0) circle[radius=1];
            \filldraw [gray] (4,0) circle[radius=0.075] node[anchor=north west] {$E'$};
            \draw[thick, <->] (3,-3) -- (3,3) node[anchor=south] {$\ell$};
            \draw[<->] (4,-3) -- (4,3) node[anchor=south] {$\ell'$};
            \draw (3,0.2) -- (3.2,0.2) -- (3.2,0);
            \draw (4,0.2) -- (4.2,0.2) -- (4.2,0);
            \draw[<->] (3.75,-3) -- (2.25,3); 
            \filldraw [gray] (3.5,-2) circle[radius=0.075] node[anchor=north east] {$C$};
            \draw [dashed] (3.5,-2) circle[radius=0.5];
            \filldraw [gray] (4,-2) circle[radius=0.075] node[anchor=north west] {$C'$};
            \draw[<->] (0,-2) -- (4.5,-2);
            \draw[<->] (20/17,22/17) -- (52/17,-38/17); 
        \end{tikzpicture} \hfill
        \begin{tikzpicture}
            \filldraw [gray] (0,3) circle[radius=0.001] node[anchor=north west] {I};
            \filldraw [gray] (0,-3) circle[radius=0.001] node[anchor=south west] {II};
            \filldraw [gray] (4.8,-3) circle[radius=0.001] node[anchor=south east] {III};
            \filldraw [gray] (4.8,3) circle[radius=0.001] node[anchor=north east] {IV};
            \draw[thick, <->] (0,0) -- (4.5,0) node[anchor=west] {$(AE)$};
            \filldraw [gray] (2.8,0.5) circle[radius=0.075] node[anchor=north east] {$B$};
            \draw [dashed] (2.8,0.5) circle[radius=1.2];
            \filldraw [gray] (4,0.5) circle[radius=0.075] node[anchor=south west] {$B'$};
            \draw[<->] (0,0.5) -- (4.5,0.5);
            \draw (3,0.7) -- (3.2,0.7) -- (3.2,0.5);
            \filldraw [gray] (3,0) circle[radius=0.075] node[anchor=north east] {$E$};
            \draw (3,0) circle[radius=1];
            \filldraw [gray] (4,0) circle[radius=0.075] node[anchor=north west] {$E'$};
            \draw[thick, <->] (3,-3) -- (3,3) node[anchor=south] {$\ell$};
            \draw[<->] (4,-3) -- (4,3) node[anchor=south] {$\ell'$};
            \draw (3,0.2) -- (3.2,0.2) -- (3.2,0);
            \draw (4,0.2) -- (4.2,0.2) -- (4.2,0);
            \draw[<->] (4.2,-3) -- (1.8,3);
            \filldraw [gray] (3.2,-0.5) circle[radius=0.075] node[anchor=south west] {$C$};
            \filldraw [gray] (4,-0.5) circle[radius=0.075] node[anchor=north west] {$C'$};
            \draw [dashed] (3.2,-0.5) circle[radius=0.8];
            \draw[<->] (56/29,-19/58) -- (4,-2.5); 
        \end{tikzpicture}
    \label{fig:2B}
\end{figure}

This argument still holds in the case where $C = C'$ is on $\ell'$, and from here moving $C$ further away from $B$ along $(BE)$ only adds the the convex hull of $b'$ and $C$, which already includes $e$. Thus all placements of $C$ on the line $(BE)$ lead to the implication $bc \to e$.

\begin{figure}[H]
    \centering
        \begin{tikzpicture}
            \filldraw [gray] (0,3) circle[radius=0.001] node[anchor=north west] {I};
            \filldraw [gray] (0,-3) circle[radius=0.001] node[anchor=south west] {II};
            \filldraw [gray] (4.8,3) circle[radius=0.001] node[anchor=north east] {IV};
            \draw[thick, <->] (0,0) -- (4.5,0) node[anchor=west] {$(AE)$};
            \filldraw [gray] (2.8,0.5) circle[radius=0.075] node[anchor=north east] {$B$};
            \draw [dashed] (2.8,0.5) circle[radius=1.2];
            \filldraw [gray] (4,0.5) circle[radius=0.075] node[anchor=south west] {$B'$};
            \draw[<->] (0,0.5) -- (4.5,0.5);
            \draw (3,0.7) -- (3.2,0.7) -- (3.2,0.5);
            \filldraw [gray] (3,0) circle[radius=0.075] node[anchor=north east] {$E$};
            \draw (3,0) circle[radius=1];
            \filldraw [gray] (4,0) circle[radius=0.075] node[anchor=north west] {$E'$};
            \draw[thick, <->] (3,-3) -- (3,3) node[anchor=south] {$\ell$};
            \draw[<->] (4,-3) -- (4,3) node[anchor=south] {$\ell'$};
            \draw (3,0.2) -- (3.2,0.2) -- (3.2,0);
            \draw (4,0.2) -- (4.2,0.2) -- (4.2,0);
            \draw[<->] (4.2,-3) -- (1.8,3);
            \filldraw [gray] (4,-2.5) circle[radius=0.075] node[anchor=south west] {};
            \filldraw [gray] (4.2,-3) circle[radius=0.075] node[anchor=south west] {$C$};
            \draw (56/29,-19/58) -- (4,-2.5);
            \draw [dashed] (56/29,-19/58) -- (4.2,-3);
            \draw [dashed] (4,0.5) -- (4.2,-3);
        \end{tikzpicture}
    \label{fig:2C}
\end{figure}

Similarly, moving $C$ horizontally off the line $(BE)$ into the opposite semi-plane as $B'$ with respect to $(BE)$, where it is assumed to be, only adds to the convex hull of $b'$ and $c'$.

\begin{figure}[H]
    \centering
        \begin{tikzpicture}
            \filldraw [gray] (0,3) circle[radius=0.001] node[anchor=north west] {I};
            \filldraw [gray] (0,-3) circle[radius=0.001] node[anchor=south west] {II};
            \filldraw [gray] (4.8,-3) circle[radius=0.001] node[anchor=south east] {III};
            \filldraw [gray] (4.8,3) circle[radius=0.001] node[anchor=north east] {IV};
            \draw[thick, <->] (0,0) -- (4.5,0) node[anchor=west] {$(AE)$};
            \filldraw [gray] (2.5,2) circle[radius=0.075] node[anchor=north east] {$B$};
            \draw [dashed] (2.5,2) circle[radius=1.5];
            \filldraw [gray] (4,2) circle[radius=0.075] node[anchor=south west] {$B'$};
            \draw[<->] (0,2) -- (4.5,2);
            \draw (3,2.2) -- (3.2,2.2) -- (3.2,2.0);
            \filldraw [gray] (3,0) circle[radius=0.075] node[anchor=north east] {$E$};
            \draw (3,0) circle[radius=1];
            \filldraw [gray] (4,0) circle[radius=0.075] node[anchor=north west] {$E'$};
            \draw[thick, <->] (3,-3) -- (3,3) node[anchor=south] {$\ell$};
            \draw[<->] (4,-3) -- (4,3) node[anchor=south] {$\ell'$};
            \draw (3,0.2) -- (3.2,0.2) -- (3.2,0);
            \draw (4,0.2) -- (4.2,0.2) -- (4.2,0);
            \draw[<->] (3.75,-3) -- (2.25,3); 
            \filldraw [gray] (3.5,-2) circle[radius=0.075] node[anchor=north east] {};
            \draw [dashed] (3.5,-2) circle[radius=0.5];
            \filldraw [gray] (4,-2) circle[radius=0.075] node[anchor=north west] {$C'$};
            \draw[<->] (0,-2) -- (4.5,-2);
            \draw[<->] (20/17,22/17) -- (52/17,-38/17); 
            \filldraw [gray] (3.2,-2) circle[radius=0.075] node[anchor=north east] {$C$};
            \draw [dashed] (3.2,-2) circle[radius=.8];
        \end{tikzpicture}
    \label{fig:2D}
\end{figure}

Therefore any possible placement of $C$ not in quadrant I leads to the strict implication $bc \to e$, so $C$ must be in quadrant I, as $B$ is. 

Now $(CE)$ splits the plane into two semi-planes, and since $E$ is in the interior of the triangle formed by $B,C$ and $D$, $D$ must be in the open semi-plane that does not contain $B$. But now we have already seen that such a placement of $D$ not in quadrant I leads to the implication $cd \to e$ (by $C$ playing the role of $B$ and $D$ playing the role of $C$ in the previous argument, so $D$ must also be in quadrant I. (Note that while $B$ was additionally shown to be in the open quadrant I, by the same argument used for $B$, $C$ also cannot be on the line $(AE)$ between $A$ and $E$, and this is all that is needed.) However, this makes it impossible for $E$ to be in the interior of the triangle formed by $B,C,$ and $D$, so there is no possible representation for the given implications without the additional implication $a \to e$.
\end{proof}

In the next statement we use the numbering of geometries given in \cite{Poly20} with their implications, but elements of geometry 74 have been relabeled to better fit the statement of Theorem \ref{OppProp}.
\begin{corollary}\label{List1}
The following geometries, identified here by their tight implications, are not representable by circles on the plane:

Geometry 46: $ab \to e$, $ac \to e$, $ad \to e$, and $bcd \to e$. 

Geometry 60: $ab \to e$, $ac \to e$, $ad \to e$, $abc \to d$, and $bcd \to e$. 

Geometry 74: $ab \to e$, $ac \to e$, $ad \to e$, and $bcd \to ae$. 

Geometry 84: $ab \to de$, $ac \to e$, $ad \to e$, and $bcd \to e$. 

Geometry 114: $ab \to de$, $ac \to de$, $ad \to e$, and $bcd \to e$. 

Geometry 115: $ab \to cde$, $ac \to e$, $ad \to e$, and $bcd \to e$.  

Geometry 153: $ab \to cde$, $ac \to de$, $ad \to e$, and $bcd \to e$.

\end{corollary}

\begin{proof}

In all of the above geometries, no proper subset of $bcd$ implies $e$, so this is a tight implication. Additionally, the implications $ab \to e$, $ac \to e$, and $ad \to e$ are all always present, so by theorem 4.1 any representation by circles on the plane would have the additional tight implication $a \to e$ not present in these geometries.
\end{proof}

\section{Area Q and non-representable geometries}

\subsection{Area \textit{Q} for binary hulls}\label{AreaQ}

In this section we consider simple facts about the convex hulls of two circles, which we call binary hulls. We investigate the possible configurations formed by two or three binary hulls involving one fixed circle $d$. The center of $d$ is denoted $D$, and similarly, we use $A,B,C$ for centers of circles $a,b,c$, respectively.

Throughout this section we will use the following labels and notation: for $u,x \in \{a,b,c,d\}$, with respect to a specified circle $u$ with center $U$, when a different circle $x$, with center $X$, is given, we will denote by $[UX)$ the ray from $U$ that goes through $X$. We denote $X_0=[UX)\cap Circ(u)$, and denote another point of intersection of the line $(UX)$ with $Circ(u)$ by $X_\infty$. Note that $X_0=X_{\infty}=U$ when $u$ is a point. 

For any center $X$ on segment $[UX_0]$, if the radius of $x$ is $\leq |XX_0|$, then $\tilde{x}\subseteq \tilde{u}$. When the radius of $x$ grows beyond $|XX_0|$, there will be exactly two common tangent lines to $u,x$ with touching points on $Circ(u)$. We will denote these two touching points $X_1,X_2$, and observe that they are symmetric around line $(UX)$.

As the radius of $x$ grows, $X_1,~X_2$ travel along corresponding halves of $Circ(u)$ from $X_0$ to $X_\infty$. When the radius reaches or surpasses $|XX_\infty|$, we obtain a configuration where $\tilde{u}\subseteq \tilde{x}$.

Similarly, for $X$ located on $[UX_0)\setminus [UX_0]$, we may start from $x=\{X\}$ and find touching points $X_1,X_2$ on $Circ(u)$ for tangents from $X$ to $u$, and as the radius of $x$ is allowed to grow, the touching points  travel toward point $X_\infty$ until $\tilde{u}\subseteq \tilde{x}$. 

For illustration with $u = d$ and $x = c$, see Fig. \ref{case2}.

Thus, whenever $\tilde{u}\not \subseteq \tilde{x}$ and $\tilde{x}\not \subseteq \tilde{u}$, there will be two lines tangent to both $u$ and $x$ with touching points $X_1,X_2$ on $Circ(u)$, symmetric around $(UX)$.

\begin{figure}[H]
\centering
\vspace{1.75mm}
\includegraphics[width=0.6\textwidth]{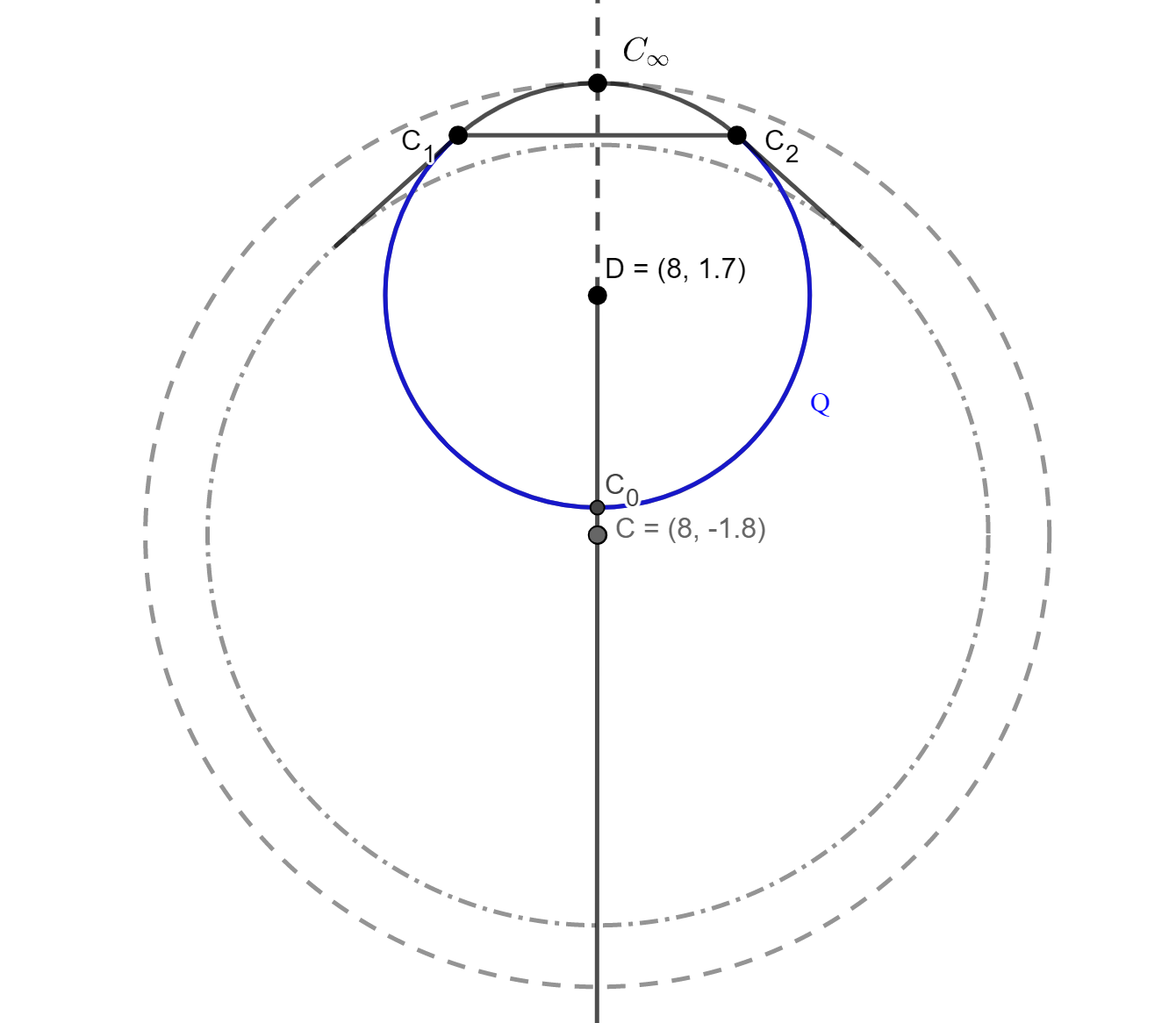}  
\caption{Location of $C_0,C_1,C_2,C_\infty$ on $Circ(d)$.}
  \label{case2}
\end{figure}

Now our goal is to establish some facts about the set of points

\[
Q_d(Y)=[\bigcap_{y\in Y} \CH(d,y)]\setminus \tilde{d},
\]
defined by some circle $d$ and set $Y$ of one, two or three additional circles. 

We start from $Y=\{c\}$, just a single circle, i.e. $Q_d(c)=\CH(d,c)\setminus \tilde{d}$.

We will assume that $c,d$ are not included within each other and thus we will have points $C_0,C_1,C_2, C_\infty$ on $Circ(d)$, which are 4 distinct points, unless $d=\{D\}$. To distinguish the two arcs connecting $C_1, C_2$ on $d$, we will include another point on the arc in the notation, in this case $\overset{\frown}{C_1C_0C_2}$ or $\overset{\frown}{C_1C_\infty C_2}$.

The following lemma links the properties of certain arcs with information about points in $Q_d(Y)$.

When we say that the set of points $S$ has points \emph{arbitrarily close to points of some set} $Z$, we mean that any open ball (standard open set of metric space $\mathbb{R}^2$) centered at $z\in Z$ has non-empty intersection with $S$.

\begin{lemma}\label{Q-one}
Suppose $\tilde{c}\not \subseteq \tilde{d}$, $\tilde{d} \not \subseteq \tilde{c}$ and $d\not = \{D\}$. Then 

\begin{enumerate}[(1)]
\item The arc $\overset{\frown}{C_1C_\infty C_2}$ is a part of the border of $\CH(c,d)$. 
\item $Q_d(c)$ has points arbitrarily close to all points of the arc $\overset{\frown}{C_1C_0C_2}$.

\end{enumerate}
\end{lemma}

\begin{proof}
(1) $\CH(c,d)$ is a convex closed set (in regular topology of $\mathbb{R}^2$), whose border consists of two tangent segments with end points $C_1,C_2$ on $Circ(d)$, the arc  $\overset{\frown}{C_1C_\infty C_2}$ of $Circ(d)$ and arc of circle $c$, see Figure \ref{case2}.

(2) The arc $\overset{\frown}{C_1C_0C_2}$ is inside closed set $\CH(c,d)$. When $d$ is removed from $\CH(c,d)$, the remaining set $Q_d(c)$ has points arbitrary close to the points of this arc.
 
\end{proof}

Now we move to consider configurations of three circles. For area Q in these cases, $Y$ consists of two circles, for example, $Y=\{b,c\}$, and $Q_d=(\CH(b,d)\cap \CH(c,d))\setminus \tilde{d}$. 

We will assume that none of the three circles involved is included into another, and in the next several statements we will refer to configurations 1, 2 or 3 from Lemma \ref{3 circles}.

\begin{lemma}\label{c1-eq}
Let circles $x,y,z$ be in configuration 1 with center $z$. Then points $X_i$ follow $Y_i$ on circle $z$ such that $\arck{X_1 X_{\infty} X_2} \subseteq \arck{Y_1 Y_{0} Y_2}$ and $\arck{Y_1 Y_{\infty} Y_2} \subseteq \arck{X_1 X_{0} X_2}$. 
\end{lemma}

\begin{figure}[H]
\centering
\vspace{1.4mm}
\includegraphics[width=0.8\textwidth]{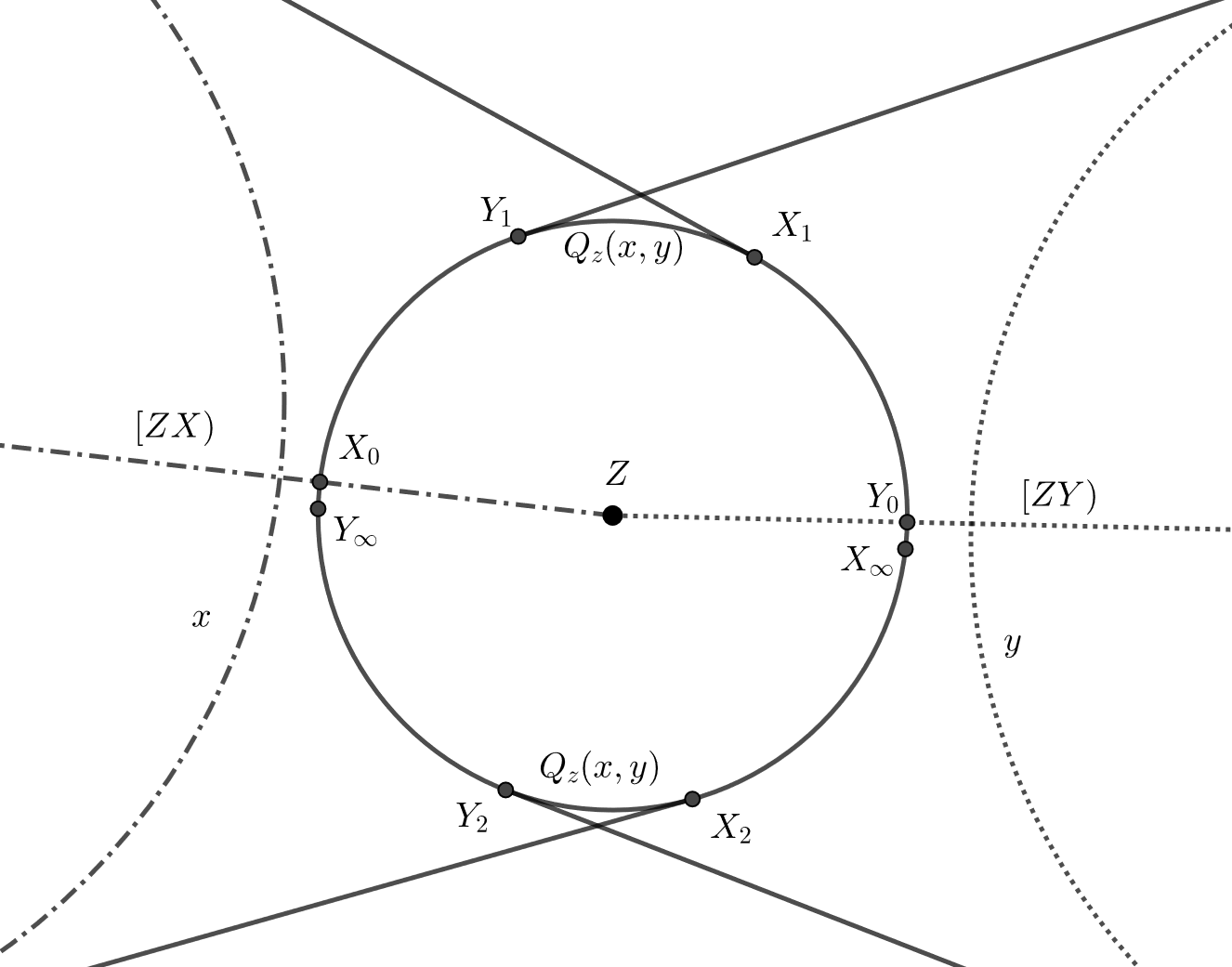}  
\caption{Lemma \ref{c1-eq}}
  \label{Lc1}
\end{figure}

\begin{proof} 
See the proof by picture on Fig. \ref{Lc1}.

\end{proof}

\begin{lemma}\label{c2-eq}
Suppose that circles $x,y,z$ are in configuration 2 with center $z$. Then:
\begin{enumerate}[label=\emph{\alph*)}]
    \item Points $X_i$ follow $Y_i$ such that
    
    $\arck{Y_1 Y_{0} Y_2} ~\subset ~\arck{X_1 X_{\infty} X_2}$ and  $\arck{X_1 X_{0} X_2} ~\subset~ \arck{Y_1 Y_{\infty} Y_2}$. 
    \item On external circle $x$,  $\arck{Y_{1}, Y_{0} Y_{2}} \subset \arck{Z_{1} Z_{0} Z_{2}}$.
    \item The border arcs of $\CH(x,y,z)$ on $z$ are formed by the intersection of arcs $\overset{\frown}{X_1 X_{\infty} X_2}$ and  $\overset{\frown}{Y_1 Y_{\infty} Y_2}$.
    \item $Q_z(x,y) = \emptyset$.
\end{enumerate}
Furthermore:
\begin{enumerate}[label=\emph{\alph*)}]
\setcounter{enumi}{3}
    \item If the circles are not in limit case 2, then $X_i \neq Y_i$
    and there are two border arcs of $\CH(x,y,z)$ on $z$.
    \item If the circles are in limit case 2, then 
    $X_1 = Y_1$ or $X_2 = Y_2$, and there is one nontrivial (i.e. consisting of more than one point)
    border arc of $\CH(x,y,z)$ on $z$.
\end{enumerate}
\end{lemma}

\begin{figure}[H]
\centering
\vspace{1.4mm}
\includegraphics[width=0.4\textwidth]{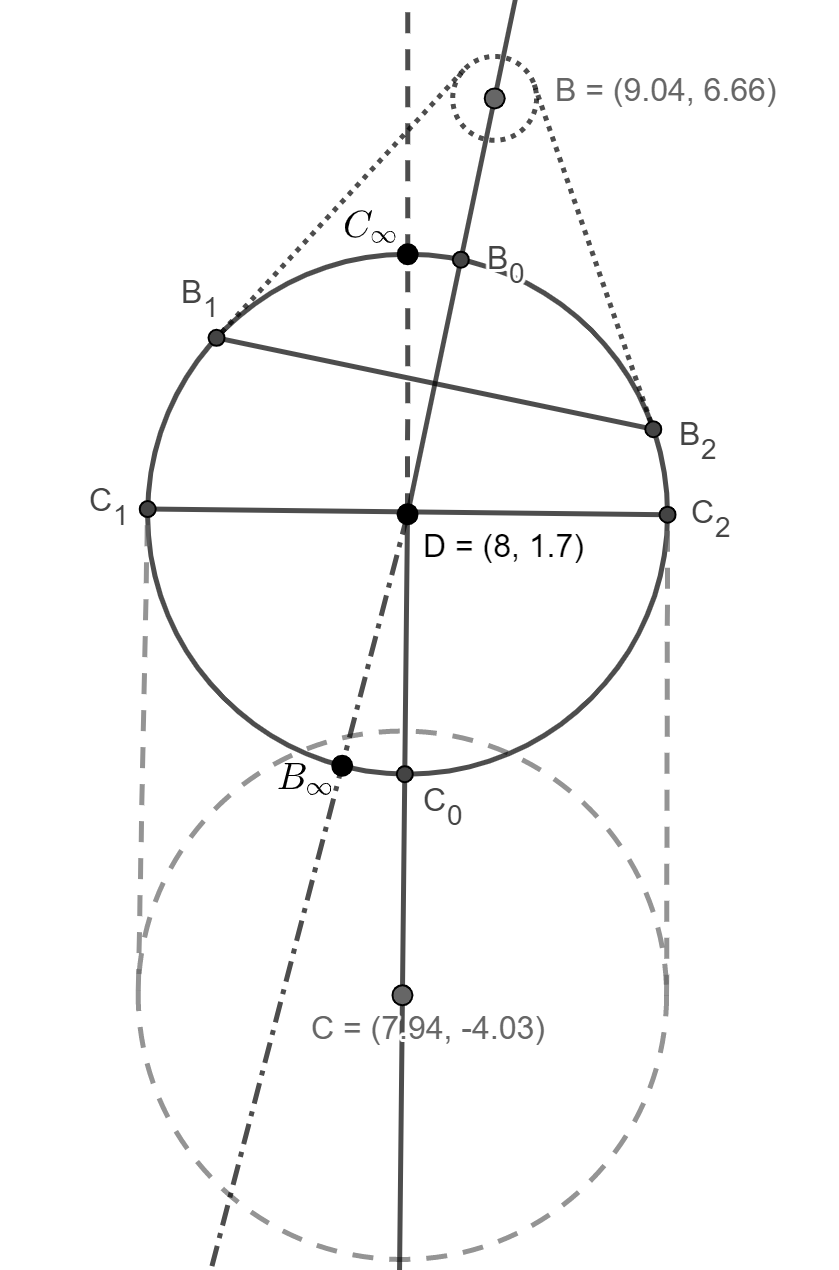}  
\caption{Lemma \ref{c2-eq}}
  \label{L2}
\end{figure}

\begin{lemma}\label{c3-eq}
Suppose that circles $x,y,z$ are in configuration 3. Then:
\begin{enumerate}[label=\emph{\alph*)}]
    \item Points $X_i$ and $Y_i$ alternate on $z$.
    \item The arc $\arck{X_1 X_0 X_2}\cap\arck{Y_1 Y_0 Y_2}$ is arbitrarily close to $Q_z(x,y)$ and $\arck{X_1 X_{\infty} X_2}\cap\arck{Y_1 Y_{\infty} Y_2}$ is part of the border of $\CH(x,y,z)$. Neither of these arcs are empty or trivial. 
\end{enumerate}
\end{lemma}

\begin{figure}[H]
\centering
\vspace{1.5mm}
\includegraphics[width=0.7\textwidth]{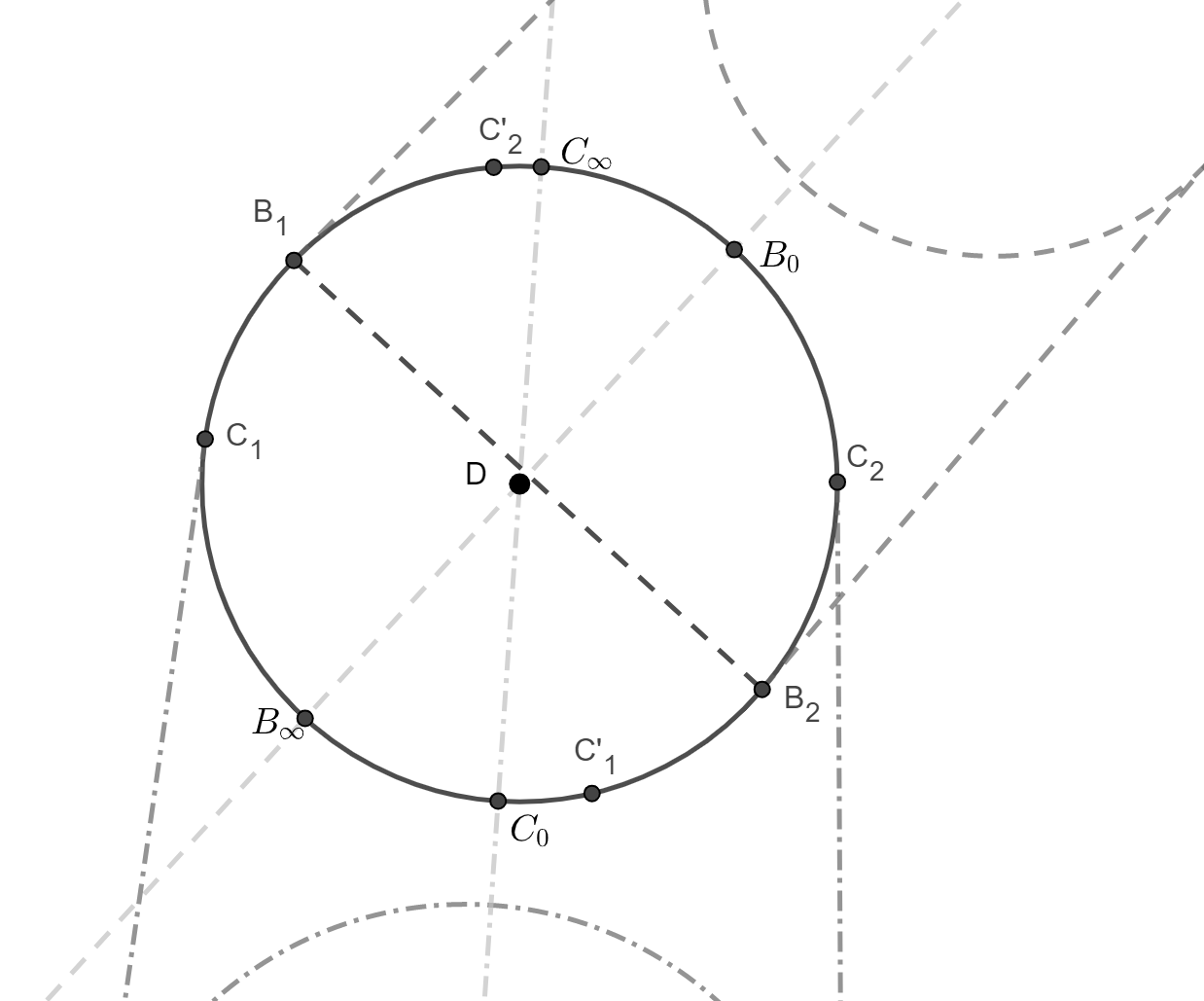}  
\caption{
Lemma \ref{c3-eq} }
  \label{L4}
\end{figure}

\begin{lemma}\label{case4}
Suppose circles $a,d,c$ are in limit case 2 configuration with $d$ in the center, and let the points on $d$ be in counter-clockwise order 
\[
A_1,A_0,A_2=C_2,C_0,C_1
\]
If the radius of $d$ is decreased, but still $d \not \in ch_c(a,c)$, then $A_2\not = C_2$ and the points are in the order
\[  
A_1, C_2, A_2, C_1
\]
\end{lemma}
\begin{proof}
Indeed, when the radius is reduced, the circles become in configuration 3, therefore, the points $A_i,C_i$ should alternate.
\end{proof}

Note that, by lemma \ref{3 circles}, any three circles on the plane with none included into another must be in one of the three configurations addressed by Lemmas \ref{c1-eq} - \ref{c3-eq}. Therefore, for such circles, one of those Lemmas must apply, and observations of the conclusions of Lemmas \ref{c1-eq} - \ref{c3-eq} can be used to determine the possible configuration(s) of such circles.

\subsection{Configurations of circles in a tight implication}

The goal of this section is to show that some circle configurations are restricted under tight implication.

We first make observations about configuration of circles $a,b,c$ and $a,c,d$ given the tight implication $abc\to d$.

\begin{lemma}\label{abc}
If $abc\to d$ is tight, then $a,b,c$ are in configuration 3.
\end{lemma}
\begin{proof}
Indeed, if $a,b,c$ is in configuration 1, then, say $c\in ch_c(a,b)$, therefore, $ab\to d$ holds, which contradicts the tightness of implication $abc\to d$. If $a,b,c$ are in configuration 2, say, with $b$ at the center, then $d\in ch_c(a,b,c)$ implies either $d\in ch_c(a,b)$ or $d\in ch_c(b,c)$, again a contradiction with the tightness.
\end{proof}

\begin{lemma}\label{NC1}
If $abc\to d$ is tight, then $a,c,d$ cannot be in configuration 1.
\end{lemma}

\begin{proof}
If $\CH(a,c,d)=\CH(a,c)$, then $d \in ch_c(a,c)$ and therefore $ac\to d$, so $abc\to d$ is not tight. If $\CH(a,c,d)=\CH(a,d)$, then $ad\to c$, and combination with $abc\to d$ requires the additional implication $ab\to cd$, as in every convex geometry, due to the anti-exchange axiom. This once again violates tightness of $abc\to d$, so $\CH(a,c,d) \neq \CH(a,d)$, and by the same argument $\CH(a,c,d) \neq \CH(c,d)$.
\end{proof}

Our next goal is to show that whenever $abc\to d$ is tight, circles $a,c,d$ will not form configuration 2, except in limit case with $d$ in the center. 

For this we show some extended version of Lemma \ref{Lemma42}.

Consider the setting: circle $c$ inscribed into angle with vertex $J$ and touching sides of angle at points $A_1,A_2$. We may adopt notation of Lemma \ref{Lemma42} 
and denote two disjoint regions of the angle outside $c$ as $w_1$ and $w_2$, and notation outlined in section \ref{AreaQ} for points $B_0,B_1,B_2$, relating to circle $b$, on circle $c$. Suppose points $B_0,B_1,B_2$ are on arc $\arck{A_1A_2}$ bordering $w_i$. 

Similar version of the statement should work for the case when two tangents of $c$ at $A_1,A_2$ are parallel lines.

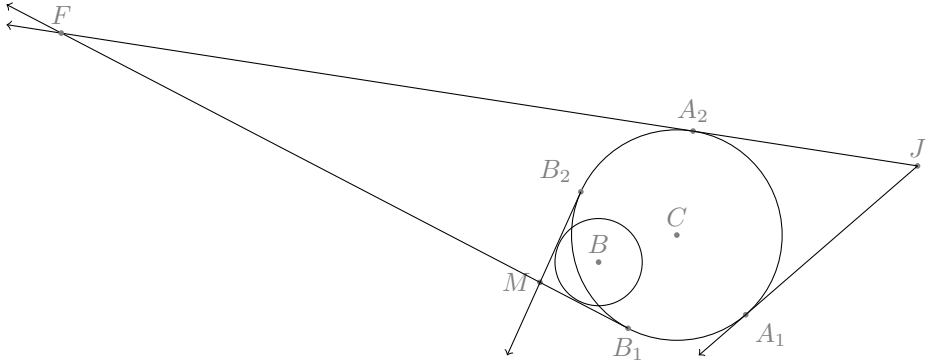
\begin{figure}[H]
    \centering
        \begin{tikzpicture}[scale=0.4]
            \draw (0,0) circle[radius=3.5];
            \filldraw [gray] (0,0) circle[radius=0.075] node[anchor=south] {$C$};
            \draw (-2.6,-0.9) circle[radius=1.45];
            \filldraw [gray] (-2.6,-0.9) circle[radius=0.075] node[anchor=south] {$B$};
            \filldraw [gray] (8,2.3) circle[radius=0.075] node[anchor=south] {$J$};
            \filldraw [gray] (-4.55,-1.57) circle[radius=0.075] node[anchor=east] {$M$};
            \filldraw [gray] (0.54,3.46) circle[radius=0.075] node[anchor=south] {$A_2$};
            \filldraw [gray] (2.29,-2.65) circle[radius=0.075] node[anchor=north west] {$A_1$};
            \draw [->] (8,2.3) -- (-22.28,7);
            \draw [->] (8,2.3) -- (0.73,-4);
            \filldraw [gray] (-3.19, 1.44) circle[radius=0.075] node[anchor=south east] {$B_2$};
            \filldraw [gray] (-1.62, -3.1) circle[radius=0.075] node[anchor=north] {$B_1$};
            \draw [->] (-1.62, -3.1) -- (-22.28,7.67);
            \filldraw [gray] (-20.46,6.72) circle[radius=0.075] node[anchor=south] {$F$};
            \draw [->] (-3.19, 1.44) -- (-5.64,-4);
        \end{tikzpicture}
    \caption{circle $b$ and $w_2$}
    \label{fig:W1}
\end{figure}

\begin{lemma}\label{W1}
If $B_0,B_1,B_2$ related to circle $b$ form a sub-arc $\arck{B_1B_0B_2}$ on arc $\arck{A_1A_2}$ of circle $c$ bordering area $w_i$, then $\tilde{b}\subseteq w_i\cup \tilde{c}$.
\end{lemma}
 
 In Figure \ref{fig:W1} circle $b$ satisfying requirement on points $B_0,B_1,B_2$ can be located inside $\CH(M,c)$, where $M$ is the intersection of tangent lines at $B_1$ and $B_2$.
 
\begin{proof}
We start from the case of infinite region $w_2$. 
 Given point $B_1$ on the arc $\arck{A_1A_2}$ bordering $w_2$, we think of ray $[B_1)$ as obtained from rotating ray $[A_1)$ that is on line $(A_1J)$ and does not have $J$: point $A_1$ moves along the arc $\arck{A_1B_1A_2}$ until it arrives at $B_1$, while the ray is tangent to circle $c$ at every location.
 
 Ray $[B_1)$ may be completely inside angle $\angle A_1JA_2$, or it may cross side $[JA_2)$. Assume that it does cross it, say, at point $F$. 
 
 Now we will be rotating ray $[A_2F)$ so that $A_2$ goes along the arc $\arck{A_2B_2A_1}$ until $A_2$ reaches $B_2$, and at every location the ray remains to be tangent to $c$. If we watch the point of intersection with line $(B_1F)$, then it moves from $F$ toward $B_1$, so it remains inside area $w_2$. In particular, the final location is point $M$, the point of intersection of two rays $[B_1)$ and $[B_2)$. 
 
 We must have $[CM) \cap Circ(c)=B_0$
 and $\CH(M,c)\subseteq w_2\cup \tilde{c}$. 
 Since any circle $b$ with touching points $B_1,B_2$ and center on $[CM)$ must be in $\CH(M,c)$, we obtain  $\tilde{b}\subseteq w_2\cup \tilde{c}$.
 
 In case none of rays $[B_1), [B_2)$ intersect with the sides of angle $\angle A_1JA_2$, we conclude that $[B_1),[B_2)\subseteq w_2$, thus, $\tilde{b}\subseteq \CH([B_1),[B_2),c)\subseteq w_2$.
 
 Similar argument works for region $w_1$, except rays $[B_1)$ and $[B_2)$, obtained by rotation of $[A_1J)$ and $[A_2J)$, respectively, always intersect with sides of angle $\angle A_1JA_2$.
\end{proof}

\begin{lemma}\label{cNotCenter}
If circle $c$ is at the center of configuration 2 for $a,c,d$ and $b,c,d$, then $d\not \in ch_c(a,b,c)$.
\end{lemma}

\begin{figure}[H]
    \centering
        \begin{tikzpicture}[scale=0.5]
            \draw (4,4) circle[radius=2.2];
            \filldraw [gray] (4,4) circle[radius=0.075] node[anchor=south] {$B$};
            \draw (4,-0.2) circle[radius=1.5];
            \filldraw [gray] (4,-0.2) circle[radius=0.075] node[anchor=south] {$A$};
            \draw (12,2) circle[radius=4.15];
            \filldraw [gray] (12,2) circle[radius=0.075] node[anchor=south] {$C$};
            \draw (17,2) circle[radius=2.15];
            \filldraw [gray] (17,2) circle[radius=0.075] node[anchor=south] {$D$};
            \filldraw [gray] (10.07,-1.67) circle[radius=0.075] node[anchor=south] {$B_2$};
            \draw (10.07,-1.67) -- (2.98,2.05);
            \filldraw [gray] (12.03,6.15) circle[radius=0.075] node[anchor=north] {$B_1$};
            \draw (12.03,6.15) -- (4.01,6.2);
            \filldraw [gray] (9.68,5.44) circle[radius=0.075] node[anchor=north] {$A_1$};
            \draw (9.68,5.44) -- (3.16,1.04);
            \filldraw [gray] (11.76,-2.14) circle[radius=0.075] node[anchor=south] {$A_2$};
            \draw (11.76,-2.14) -- (3.91,-1.7);
            \filldraw [gray] (13.66, 5.8) circle[radius=0.075] node[anchor=north] {$D_1$};
            \draw (13.66, 5.8) -- (17.86,3.97);
            \filldraw [gray] (13.66,-1.8) circle[radius=0.075] node[anchor=south] {$D_2$};
            \draw (13.66,-1.8) -- (17.86,0.03);
        
        \end{tikzpicture}
        \caption{$c$ is at the center in $b,c,d$}
    \label{fig:abcd}
\end{figure}
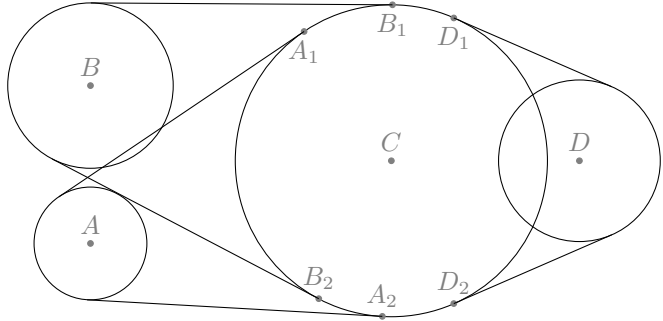

\begin{proof}
 Note that $c$ cannot have radius 0 in the described configurations.
 
 Let $D_1,D_2,D_0,D_\infty$ on $c$ be as specified in section \ref{AreaQ}.

 Then $A_1A_2,B_1,B_2$, the touching points on $Circ(c)$ from circles $a,b$, do not alternate with $D_1,D_2$; thus, all are located on arc $\arck{D_1D_\infty D_2}$.
 By assumption on configurations, points $A_0$ and $B_0$ are also on the same arc. 
 
 Let us assume wlog that the region between tangents bordered by this arc is $w_2$.  Then by Lemma \ref{W1} we have $\tilde{a},\tilde{b}\subseteq w_2\cup \tilde{c}$. On the other hand, $\tilde{d}\subseteq w_1\cup\tilde{c}$, where we call $w_1$ the region formed by extending tangent segments between $c,d$ beyond touching points on $d$, and bordered by arc $\arck{D_1D_0D_2}$. Since $w_1,w_2$ are disjoint regions, $d$ cannot be in $ch_c(a,b,c)$ (unless $d$ is inside $c$, which we exclude).  
\end{proof}

\begin{lemma}\label{conf2}
Let circles $a,b,c$ form configuration 3 and $abc\to d$. If $a,c,d$ form configuration 2 with center $c$, then $bc\to d$.
\end{lemma}

\begin{proof}
Let $A_0,A_\infty, D_0,D_\infty$ on $c$ be as specified in section \ref{AreaQ}.

Since $c$ is in the center of $a,c,d$ in configuration 2, touching points $A_i,D_i$ on $c$ do not alternate (and $c$ has a positive radius). 
Then $D_1,D_2$ are located on arc $\arck{A_1A_\infty A_2}$.
 
 Since $a,b,c$ are in configuration 3, touching points $A_i,B_i$ alternate on $c$. Without loss of generality assume that $\arck{A_1B_1}$ is an arc where $Q_c(a,b)$ is attached, then $\arck{A_2B_2}$ is the arc on $c$ which is part of the border of $\CH(a,b,c)$. Due to the latter fact, arc $\arck{A_2B_2}$ cannot be on $\arck{A_1A_0A_2}$, which is inside $\CH(a,c)$, therefore, arc  $\arck{A_2B_2}$ is on $\arck{A_1A_\infty A_2}$.
 Apparently, $D_1,D_2$ cannot be on arc $\arck{A_2B_2}$, therefore, they are on  arc $\arck{A_1A_\infty A_2}$
 from which we remove arc $\arck{A_2B_2}$. This is arc $\arck{B_2A_1}$, which is also a part of arc $\arck{B_2A_1B_1}$. We conclude that $D_1,D_2$ are also on the arc $\arck{B_2A_1B_1}$, therefore, $B_i,D_i$ do not alternate.
 
 This means that $b,c,d$ are in configuration 1 or 2. If $d$ is in the center of configuration 1, then $bc\to d$ and we are done. If, say, $b$ is in the  center of configuration 1, then $cd\to b$, thus, together with $abc\to d$ we conclude $ac\to bd$ due to the anti-exchange property, and $ac\to d$ contradicts the assumption that $c$ is in the center of $a,c,d$ in configuration 2.
 
 It remains to show that $b,c,d$ in configuration 2 will bring $bc\to d$ or a contradiction with $abc\to d$ or configuration 3 for $a,b,c$. 
 
 If $c$ is at the center of configuration 2 for $b,c,d$, then we have the setting of Lemma \ref{cNotCenter}, and we get to the contradiction with $abc\to d$. See Figure \ref{fig:abcd}.
 
 \begin{figure}[H]
    \centering
        \begin{tikzpicture}[scale=0.5]
            \draw (18.2,3.1) circle[radius=0.8];
            \filldraw [gray] (18.2,3.1) circle[radius=0.075] node[anchor=south] {$B$};
            \draw (2,2.9) circle[radius=0.75];
            \filldraw [gray] (2,2.9) circle[radius=0.075] node[anchor=south] {$A$};
            \draw (10,2.75) circle[radius=3.25];
            \filldraw [gray] (10,2.75) circle[radius=0.075] node[anchor=south] {$C$};
            \draw (13.5,3.2) circle[radius=3.05];
            \filldraw [gray] (13.5,3.2) circle[radius=0.075] node[anchor=south] {$D$};
            \filldraw [gray] (11.1,-0.31) circle[radius=0.075] node[anchor=north west] {$B_2$};
            \draw (11.1,-0.31) -- (18.47,2.35);
            \filldraw [gray] (10.84,5.89) circle[radius=0.075] node[anchor=south] {$B_1$};
            \draw (10.84,5.89) -- (18.41,3.87);
            \filldraw [gray] (9.04,5.86) circle[radius=0.075] node[anchor=south east] {$A_1$};
            \draw (9.04,5.86) -- (1.78,3.62);
            \filldraw [gray] (8.93,-0.32) circle[radius=0.075] node[anchor=north] {$A_2$};
            \draw (8.93,-0.32) -- (1.75,2.19);
            \filldraw [gray] (9.77,5.99) circle[radius=0.075] node[anchor=south] {$D_1$};
            \draw (9.77,5.99) -- (13.28,6.24);
            \filldraw [gray] (10.6,-0.44) circle[radius=0.075] node[anchor=north] {$D_2$};
            \draw (10.6,-0.44) -- (14.06,0.2);
        
        \end{tikzpicture}
    \caption{$d$ is at the center in $c,d,b$}
    \label{fig:acdb}
\end{figure}
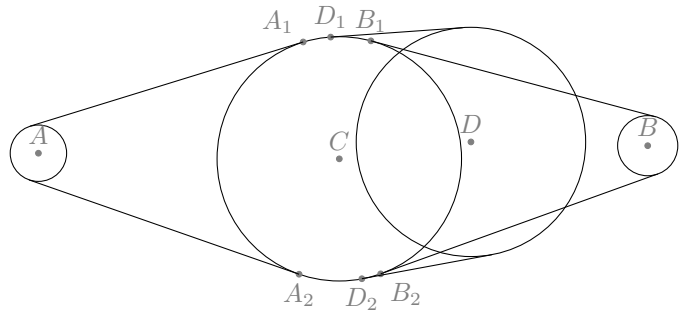

 Suppose $c,d,b$ are in configuration 2 and $d$ is in the center. Then on circle $c$ points $B_1B_2$ are located on the arc $\arck{D_1D_0D_2}$, while points $A_1,A_2$ are located on arc $\arck{D_1D_\infty D_2}$.
  
 Indeed, the latter is due to $a,c,d$ also in configuration 2 with $c$ in the center. Thus, $A_i,B_i$ do not alternate on $c$, a contradiction with assumption that $a,b,c$ in configuration 3. See Figure \ref{fig:acdb}.
 
 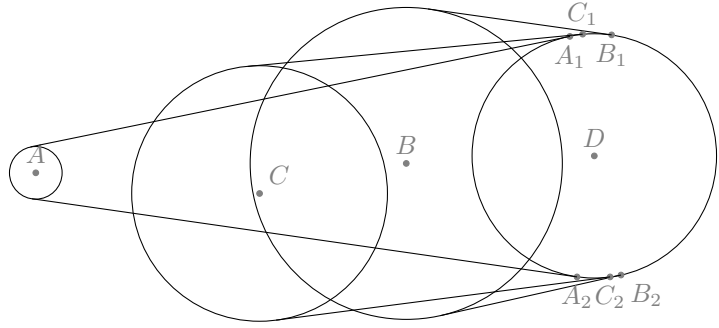
\begin{figure}[H]
    \centering
        \begin{tikzpicture}[scale=0.5]
            \draw (12.2,2) circle[radius=4.15];
            \filldraw [gray] (12.2,2) circle[radius=0.075] node[anchor=south] {$B$};
            \draw (2.35,1.75) circle[radius=0.7];
            \filldraw [gray] (2.35,1.75) circle[radius=0.075] node[anchor=south] {$A$};
            \draw (8.3,1.2) circle[radius=3.4];
            \filldraw [gray] (8.3,1.2) circle[radius=0.075] node[anchor=south west] {$C$};
            \draw (17.2,2.2) circle[radius=3.25];
            \filldraw [gray] (17.2,2.2) circle[radius=0.075] node[anchor=south] {$D$};
            \filldraw [gray] (17.91,-0.97) circle[radius=0.075] node[anchor=north west] {$B_2$};
            \draw (17.91,-0.97) -- (13.11,-2.05);
            \filldraw [gray] (17.66,5.42) circle[radius=0.075] node[anchor=north] {$B_1$};
            \draw (17.66,5.42) -- (12.78,6.11);
            \filldraw [gray] (16.55,5.38) circle[radius=0.075] node[anchor=north] {$A_1$};
            \draw (16.55,5.38) -- (2.21,2.44);
            \filldraw [gray] (16.74,-1.02) circle[radius=0.075] node[anchor=north] {$A_2$};
            \draw (16.74,-1.02) -- (2.25,1.06);
            \filldraw [gray] (16.89,5.44) circle[radius=0.075] node[anchor=south] {$C_1$};
            \draw (16.89,5.44) -- (7.98,4.58);
            \filldraw [gray] (17.62,-1.02) circle[radius=0.075] node[anchor=north] {$C_2$};
            \draw (17.62,-1.02) -- (8.74,-2.17);
        
        \end{tikzpicture}
    \caption{$b$ is at the center in $c,b,d$}
    \label{fig:acbd}
\end{figure}

 Finally, consider $b$ at the center of $c,b,d$, like in Figure \ref{fig:acbd}. For now, assume $d$ has a positive radius. Then on circle $d$ points $B_iC_i$ do not alternate
 , and $C_1,C_2,C_0$ are on the arc $\arck{B_1B_0B_2}$. But also, since $a,c,d$ are in configuration 2 with $c$ at the center, $A_1,A_2,A_0$ on $d$  are on the arc $\arck{C_1C_0C_2}$. Apply Lemma \ref{W1} replacing $c$ by $d$, points $A_1,A_2$ by $B_1,B_2$, lines $(A_1J),(A_2J)$ by tangents of $b,d$ and $B_1,B_2,B_0$ by either $A_1,A_2,A_0$ or by $C_1,C_2,C_0$. Then, circles $a,b,c$ are inside $w_i\cup \tilde{d}$, where $w_i$ is a region formed by tangents of $d$ at $B_1,B_2$ and bordered by arc $\arck{B_1B_0B_2}$. This implies that $d\not \in ch_c(a,b,c)$ (unless $b=d$, which we exclude). 
 
 The case when $d$ is a point follows from the above: if $d$ is not on a tangent of binary hulls within $\CH(a,b,c)$, we can pick some small radius for a circle $d'$ with center $d$ such that the assumptions hold ($abc ->d'$,$a,c,d'$ are in configuration 2 with center $c$, and $c,b,d'$ are in configuration 2 with $b$ in the center) and lead to a contradiction as before. If $d$ is on a tangent of $\CH(b,c)$ then the conclusion holds, and if $d$ is on a tangent of $\CH(a,b)$ or $\CH(a,c)$, we can replace both $d$ and $a$ with slightly larger circles such that the given assumptions hold and are again in the case above. 
\end{proof}

We finally collect all the previous observations in the following result.

\begin{theorem}\label{NC2}
If $abc\to d$ is tight, then $a,c,d$ is in configuration 3 or in limit case 2 with $d$ at the center.
\end{theorem}

\begin{proof}
By Lemma \ref{NC1}, circles $a,c,d$ cannot be in configuration 1.
By Lemma \ref{abc} $a,b,c$ form configuration 3, therefore, the assumptions of Lemma \ref{conf2} hold. By Lemma \ref{conf2} if any $a$ or $c$ is in the center of configuration 2 for $a,c,d$, then the implication $abc\to d$ is not tight.

Thus, if $a,c,d$ form configuration 2, then only $d$ can be in the center. Moreover one of common tangents of $a,c$ has the border of $\CH(a,b,c)$, therefore, $d$ cannot intersect it in more than one point. Thus, if $d$ is in center of configuration 2 for $a,c,d$, then it is in limit case 2.
\end{proof}

\subsection{Point Order Theorem}
We will be establishing consistent labeling of touching points on circle $d$, when four circles satisfy tight implication  $abc \to d$, for the  underlying closure operator $ch_c$.

We first fix a consistent orientation of circles $a, b, c$ around $d$. 
We will set points on $Circ(d)$ using the approach described in section \ref{AreaQ}. Let ray $[DA)$ start at center $D$ and go through center $A$. The point of crossing $Circ(d)$ and  $[DA)$ is called $A_0$. 

The second point of intersection of $Circ(d)$ and $(AD)$ is denoted $A_{\infty}$. The same labeling conventions give $B_0$, $B_{\infty}$ from ray $[DB)$ and circle $b$. Furthermore, points $C_0$ and $C_{\infty}$ come from $[DC)$ and $c$. 

Now we establish the orientation. Consider circle $d$ as a clock with the center of rotation for the clock hand at $D$. Start the clock hand at $B_0$ and move counterclockwise. Without loss of generality, the circular order on $Circ(d)$ will be as follows:
$B_0,A_0,C_0$.

Now we establish the orientation of tangent points with respect to the various ray intersections on $d$. There are two tangent points on $d$ from $\CH(b,d)$ that we label $B_1$ and $B_2$. When starting at $B_1$ and walking counterclockwise, the circular order will be  $B_1, B_0, B_2$. The tangent points from $\CH(a,d)$ are $A_1$ and $A_2$ with $A_0$ inside $\CH(a,d)$. Starting at $A_1$, we have the circular order as $A_1,A_0,A_2$. For $\CH(c,d)$, we have points $C_1, C_2$. When walking counterclockwise, we get the point order $C_2, C_0, C_1$.

Finally, we label the outer tangents of $\CH(a,b,c)$ as $(A_bB_a)$, $(A_cC_a)$ and $(B_cC_b)$.

For any two non-collinear rays $[DA)$ and $[DC)$ we call \emph{proper angle} a smaller of two angles formed by these two rays.

\begin{theorem}[Point Order Theorem]\label{POrd}
With the counterclockwise orientation of $B_0,A_0,C_0$ and the associated labeling when $abc \to d$ is tight, the counterclockwise point order on $d$ beginning at $B_1$ is  $B_1, C_1, A_1, B_2, C_2, A_2$.
\end{theorem}

\begin{proof}
By Theorem \ref{NC2}, triples of circles: $\{a,c,d\}$, $\{b,c,d\}$, and $\{a,b,d\}$, are only in limit case 2 with $d$ at the center or the convex hulls are in configuration 3. Since $d$ is in $\CH(abc)$, we may start with the case where $d$ is touching all three lines  
$(A_bB_a)$, $(A_cC_a)$ and $(B_cC_b)$.

Then we proceed with cases where $d$ is tangent to two of these lines, one of these lines, and finally none of these lines.

\textbf{\textit{Case 1:} $d$ is inscribed in $\CH(a,b,c)$.}

\begin{figure}[H]
\centering
\vspace{1.5mm}
 \includegraphics[scale = 0.8]{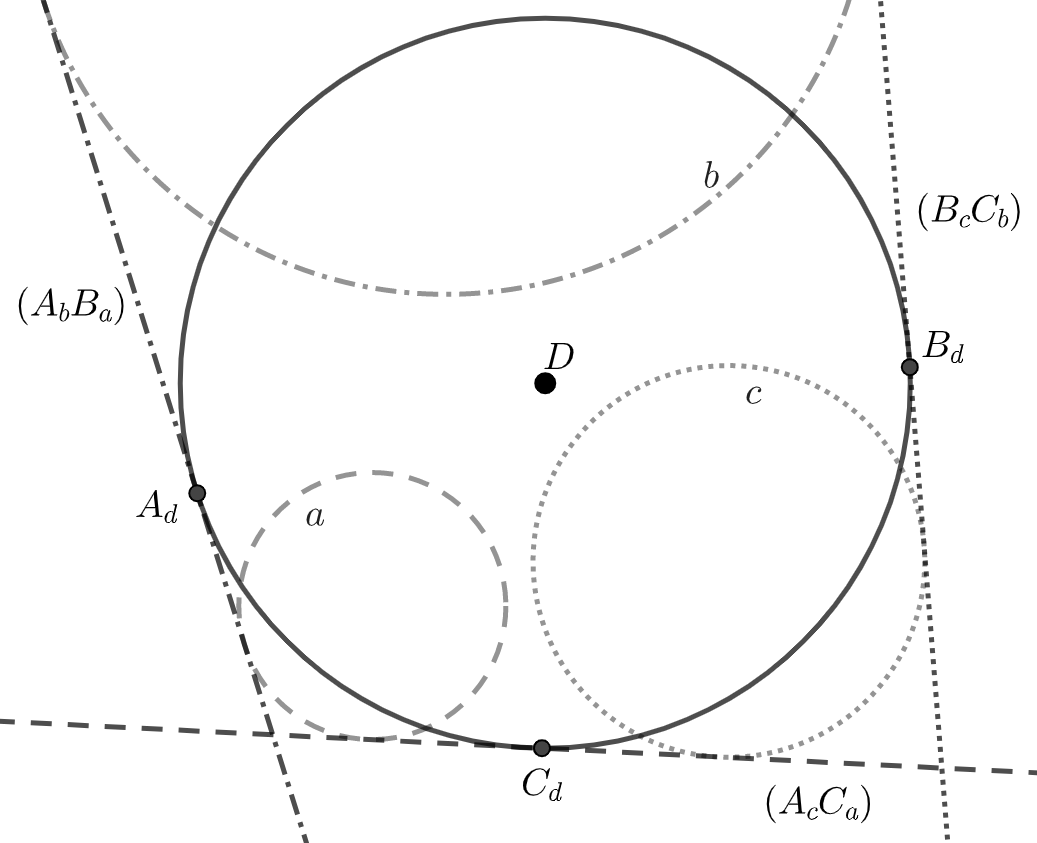}
 \caption{Theorem \ref{POrd}}
  \label{N1}
\end{figure}

This is the case where $d$ is tangent to $(A_bB_a)$, $(A_cC_a)$, and $(B_cC_b)$. For discussion convenience, label the tangent points on $d$ as $B_d$,$A_d$ and $C_d$, on lines $(B_cC_b)$, $(A_cC_a)$ and $(B_cC_b)$, respectively.

Then, according to agreement on orientation of $a,b,c$ around $d$, we have the following coincidence of points:
\[
B_1=C_1=B_d,\ \   A_1=B_2 = A_d, \ \ C_2= A_2 = C_d
\]

Therefore, the counter-clock sequence is, indeed, 
\[
B_1, C_1=B_1, A_1, B_2=A_1, C_2, A_2=C_2.
\]

\textbf{\textit{Case 2:} $d$ is tangent to $(A_bB_a)$ and $(A_cC_a)$ but not $(B_cC_b)$ }.

This new condition 
does not change the assumption that $A_1=B_2= A_d$ and $C_2=A_2 = C_d$. 

Now $b,c,d$ are in configuration 3 by Theorem \ref{NC2}, which means $B_i, C_i$ alternate around $d$ by Lemma 
\ref{c3-eq}. 
This allows the counter-clockwise point order $B_1, C_1, B_2, C_2,$ or, the order $B_1, C_2, B_2, C_1$. The latter order is not possible, because circles $a,c,d$ are in limit case 2 with center $d$, therefore, 
by Lemma \ref{c2-eq}, $\arck{A_1 A_{0} A_2} ~\subset ~\arck{C_1 C_{\infty} C_2}=~\arck{C_1 C_{\infty} A_2}$. 
But with the order $B_1, C_2, B_2, C_1$, the arc $\arck{C_2C_{0}C_1}$ will  contain point $B_2=A_1$ instead of the arc $\arck{C_2C_{\infty}C_1}$, a contradiction.

Therefore, the only possible order is  $B_1, C_1, B_2, C_2,$, and, thus, including $A_1,A_2$ into the sequence, we get
\[
B_1,C_1,A_1=B_2,B_2,C_2,A_2=C_2.
\]

\textbf{\textit{Case 3:} $d$ is tangent to $(A_cC_a)$ but not $(B_cC_b)$ and $(A_bB_a)$}.

We still have $C_2=A_2=C_d$, thus $d$ is at the center of limit case 2 for circles $a,c,d$. Therefore, the circular sequence of points on $d$  starting from $C_2=A_2$ will be either
\[
C_2, C_0, C_1, A_1, A_0, A_2=C_2
\]
or

\[
A_2,A_0,A_1,C_1,C_0,C_2=A_2
\]

But only first sequence satisfies the agreement of counter-clockwise order $A_1,A_0,A_2$.

Now, by assumption in this case, circles $a,b,d$ as well as $b,c,d$ are in configuration 3. Since points $B_i$ alternate with both $A_i$ and $C_i$, by Lemma \ref{c3-eq}, one of points $B_i$ is on arc $\arck{C_2C_0C_1}$ and another on arc $\arck{A_2A_0A_1}$. 
There are two possibilities to insert points $B_i$ into the first sequence above. Suppose we have the following placement of $B_i$:
\[
C_2, B_2, C_1, A_1, B_1,A_2=C_2
\]

This would mean that the arc $\arck{B_1B_0B_2}$ is the same as $\arck{B_1A_2B_2}$, but arc $\arck{B_1B_0B_2}$ is an inner arc of 
$\CH(b,d)$, with only points $B_1,B_2$ on its border, thus, it is an inner arc of $\CH(a,b,c)$, and cannot have points on the border of $\CH(a,b,c)$, in particular, cannot have point $A_2=C_2$. 

Therefore, the only possible sequence of points is 

\[
C_2, B_1, C_1, A_1, B_2,A_2=C_2
\]
which can be written starting from $B_1$ as follows:

\[
 B_1, C_1, A_1, B_2, C_2, A_2=C_2
 \]

\textbf{\textit{Case 4:} $d$ is not tangent to $(A_cC_a)$, $(B_cC_b)$ and $(A_bB_a)$}.

Now every triple of circles: $\{a,b,d\}$, $\{b,c,d\}$ and $\{a,c,d\}$ - is in configuration 3 by Theorem \ref{NC2}, therefore, $A_i,B_i$ alternate on $d$, as well as $B_i,C_i$ and $A_i,C_i$.

Gradually increase the radius of $d$ so that $d$ remains in $\CH(a,b,c)$. At the limit of this process $d$ will touch at least one of outer borders of $\CH(a,b,c)$. Call new circle $d_1$.
Then $abc\to d_1$ is again tight. Moreover, one of Cases 1-3 will describe the configuration of $a,b,c,d_1$.

Assume wlog that $d_1$ is touching $(A_cC_a)$ and not touching other lines $(A_bB_a)$, $(B_cC_b)$. Then we can apply the result of Case 3 to $d_1$. Therefore, the sequence of points on $d_1$ is
\[
 B_1, C_1, A_1, B_2, C_2=A_2
\]

Also, focusing on subsequence of $A_i,C_i$, we have, starting from $A_1$ :
\[
A_1,A_0,A_2=C_2,C_0,C_1
\]

If we slightly reduce the radius of $d_1$, circles $a,c,d_1$ will no longer be in limit case 2, but since $abc\to d_1$ is tight, $a,c,d_1$ are in configuration 3, due to Theorem \ref{NC2}. By Lemma \ref{case4}, the sequence of $A_i,C_i$ points on $d_1$ starting from $A_1$ is
\[
A_1, C_2, A_2,C_1
\]

We need to embed this sequence into sequence of limit case 2
\[
 B_1, C_1, A_1, B_2, A_2=C_2
\]

We notice that we have to place $C_2$ after $A_1$, and if we place it before $B_2$, then points $B_i,C_i$ will not alternate, which contradicts $b,c,d_1$ being in configuration 3. Therefore, the proper sequence is 
\[
 B_1, C_1, A_1, B_2, C_2, A_2
\]

If we continue reducing the radius of $d_1$ back to $d$, we remain in configuration 3 for all triples of circles. Indeed, at every moment of the process we have $abc\to d_1$ tight, and at the start of the process $d_1$ was not at the center of limit case 2 with any two other circles. Therefore, the sequence of points on $d$ will not change.

Suppose we get $d_1$ in its limit position touching $(A_bB_a)$ and $(A_cC_a)$. Then, according to the result for case 2, we get the following sequence on $d_1$:
\[
B_1,C_1,A_1=B_2,B_2,C_2,A_2=C_2
\]

When the radius of $d_1$ is slightly reduced, we get the following sequence of points $B_i,A_i$:
\[
B_1, A_1, B_2, A_2
\]
Similarly, we expect the following sequence of points $A_i,C_i$:
\[
C_1,A_1, C_2, A_2
\]

Incorporating them into original sequence in limit position of $d_1$, we get only one possibility:
\[
B_1, C_1, A_1, B_2, C_2, A_2
\]

Finally, consider the case when $d_1$ in its limit position is touching all three lines $(A_cC_a)$, $(B_cC_b)$ and $(A_bB_a)$.
Then we have configuration of $a,b,c,d_1$ in case 1, therefore, there are only three distinct touching points on $d_1$:
\[
B_1, A_1, C_2
\]
If we give them other labeling, focusing on $A_i,B_i$, then we get the same sequence as
\[
B_1, B_2=A_1, A_2
\]
Then, when the radius of $d_1$ is slightly reduced, we expect, by Lemma \ref{case4}, that the points will alternate:
\[
B_1,A_1, B_2, A_2
\]

Similarly, if we write the original sequence of points on $d_1$ in their $A_i,C_i$ labeling, we get
\[
C_1, A_1, C_2=A_2
\]
Therefore, after reducing the radius of $d_1$ we obtain
\[
C_1,A_1,C_2,A_2
\]
Finally, if we write the original sequence in its $B_i,C_i$ labeling, we will have
\[
B_1=C_1, B_2, C_2
\]
After the radius of $d_1$ is reduced it turns into
\[
B_1,C_1,B_2,C_2
\]

There is only one circular sequence of all points $A_i,B_i,C_i$ that honors three established subsequences:
\[
B_1, C_1, A_1, B_2, C_2, A_2
\]
This finishes the analysis of Case 4.
\end{proof}

\section{Point Order Theorem in action}

In this section we find several applications of Point Order Theorem, which will generate several families of non-representable convex geometries.

\subsection{Separation property}

\begin{theorem}\label{G14}
Let $a,b,c,d,e$ be circles on the plane. If $acd \to e$ and $bcd \to e$ are both tight, then it is not possible to also have the implication $ab \to e$.
\end{theorem}
\begin{proof}
First, let us consider the case where $e=\{E\}$. Since $acd \to e$ and $bcd \to e$ are both tight, by the triangle property (Lemma \ref{triangle}) the center of circle $e$ must lie in the interior of the triangles formed by centers $B, C, D$ and $A, C, D$. Then, by lemma \ref{OpAngle}, both $B$ and $A$ must be in the opposite angle to $\angle DEC$. Thus, if any point of $a$ crosses the line $(DE)$, then $E \in \CH(a,D) \subseteq \CH(a,d)$, violating the tightness of the implication $acd \to e$. (See figure \ref{G14pt} below.) 

Similarly, if any point of $a$ crosses the line $(CE)$, then $E \in \CH(a,c)$, so $a$, and by the same argument $b$, must be strictly contained in the opposite angle to $\angle DEC$, and therefore their convex hull cannot contain $e$ and the hypothesis holds.

\begin{figure}[H]
    \centering
        \begin{tikzpicture}
            \draw[<->] (-1,-1) -- (2,1);
            \filldraw [gray] (-1,-1) circle[radius=0.075] node[anchor=north] {$C$};
            \draw[<->] (1,-1) -- (-1,1);
            \filldraw [gray] (1,-1) circle[radius=0.075] node[anchor=north] {$D$};
            \filldraw [gray] (1/5,-1/5) circle[radius=0.075] node[anchor=north] {$E$};
            \draw[opacity=0.2,fill=gray] (2,1) -- (1/5,-1/5) -- (-1,1) -- (2,1);
            \draw [dashed] (-0.2, 0.4) circle[radius=0.4];
            \filldraw [gray] (-0.2, 0.4) circle[radius=0.075] node[anchor=south] {$A$};
        
        \end{tikzpicture}
    \caption{Possible locations for $A,B$ shaded}
    \label{G14pt}
\end{figure}
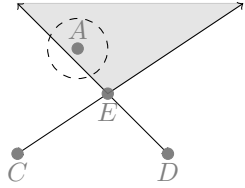

Now let us consider the case when $e$ has a positive radius. Wlog, let us assume the counterclockwise orientation of $B_0, D_0, C_0$ and associated labeling around circle $e$ as in Theorem \ref{POrd}, with circles $d$ and $e$ playing the roles of circles $c$ and $d$ respectively. Specifically, note that we have labeled such that $B_1, B_0, B_2$ is the counterclockwise circular order of these points around $e$. Since $bcd \to e$ is tight, this Lemma then tells us that the counterclockwise point order on $e$ beginning at $B_1$ is  $B_1, C_1, D_1, B_2, C_2, D_2$.

Now, consider the placement of $A$, the center of circle $a$. Again, since $acd \to e$ and $bcd \to e$ are both tight, by the triangle property the center of circle $e$ must lie in the interior of the triangles formed by centers $B, C, D$ and $A, C, D$, and by lemma \ref{OpAngle}, both $B$ and $A$ must be in the opposite angle to $\angle DEC$. 

In particular, the rays from $E$ to $A$ and $B$ that define $A_0$ and $B_0$ are both in the opposite angle to $\angle DEC$, and so the counterclockwise orientation of $B_0, D_0, C_0$ around $e$ must be the same as the counterclockwise  orientation of $A_0, D_0, C_0$. Thus, again using the same labeling such that $A_1, A_0, A_2$ is the counterclockwise circular order of these points around $e$, by Theorem \ref{POrd} the counterclockwise point order on $e$ beginning at $A_1$ is  $A_1, C_1, D_1, A_2, C_2, D_2$.

Now, for the sake of contradiction, suppose we also have the implication $ab \to e$. Then circles $a,b,e$ are in configuration 1, with $e$ in the center, so by lemma \ref{c1-eq} $A_1, A_2$ follow $B_1, B_2$ on a walk around $e$. The only possible point orders are then $A_1, B_1, C_1, D_1, B_2, A_2, C_2, D_2$, or the symmetric order with $a$ and $b$ switched. In this case, $\overset{\frown}{B_1B_0B_2} \subseteq \overset{\frown}{A_1A_0A_2}$, so the complementary arc $\overset{\frown}{B_1B_\infty B_2} \not\subseteq \overset{\frown}{A_1A_0A_2}$ (unless $\overset{\frown}{A_1A_0A_2}$ is all of circle $e$, which we exclude since $a \to e$ contradicts $acd \to e$ being tight), which contradicts Lemma \ref{c1-eq}. Thus the implication $ab \to e$ must be absent.

\end{proof}

In the following corollary, the labels in geometry 26 have been changed to better fit the statement of theorem \ref{G14}.

\begin{corollary}\label{List4}
The following geometries, identified here by their tight implications, are not representable by circles on the plane:

Geometry 14: $acd \to e$, $bcd \to e$, and  $ab \to e$. 

Geometry 23: $acd \to e$, $bcd \to e$, $abc \to de$, and  $ab \to e$. 

Geometry 26: $acd \to be$, $bcd \to e$, and  $ab \to e$. 

Geometry 39: $acd \to e$, $bcd \to e$, and  $ab \to de$. 

Geometry 54: $acd \to e$, $bcd \to e$, and  $ab \to cde$.

\end{corollary}
\begin{proof}
In all these geometries the assumptions of Theorem \ref{G14} hold, but the conclusion is not compatible with the given implications.  
\end{proof}

\subsection{Nested triangles}

Finally, we will make use of technical work of the previous section. The first statement is a key Lemma that deals with the tight implication $abc\to d$. We could think of the combination of ``nested triangles" formed by $d,a,b$ and $d,b,c$ (or $d,a,c$) within the larger ``triangle" formed by $a,b,c$.

 \begin{figure}[H]
\centering
\vspace{1.5mm}
 \includegraphics[scale = 0.55]{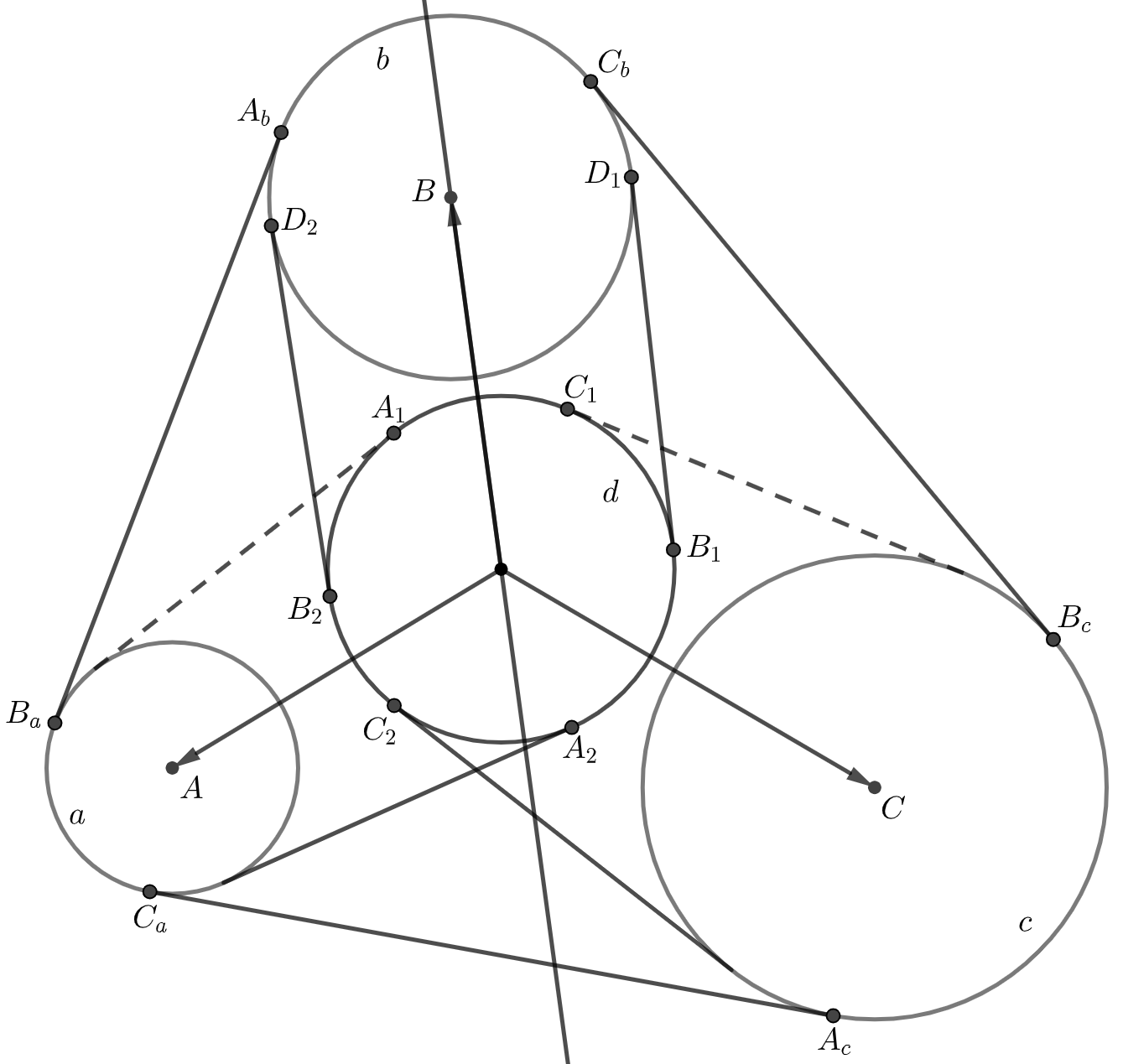}
 \caption{Lemma \ref{Nested}}
  \label{N}
\end{figure}

\begin{lemma}\label{Nested}
If $abc \to d$ is tight, then $\CH(a,b,d)\cap \CH(b,c,d) = \CH(b,d)\cup Q_d(a,c)$ in case $r_d>0$, and $\CH(a,b,d)\cap \CH(b,c,d) = \CH(b,d)$ when $r_d=0$.
\end{lemma}

\begin{proof} 

Let us assume that we proved the case $r_d>0$, and consider $r_d=0$.

Since the right side of the equality for that case is always contained in the left side, we need to show that the left side is contained in the right side. 

According to assumption, we have for  
$r_d>0$:
\[
\CH(a,b,\{D\})\cap \CH(b,c,\{D\}) \subseteq \CH(a,b,d)\cap \CH(b,c,d) = \CH(b,d)\cup Q_d(a,c)
\]

Now we can take the limit when $r_d$ approaches $0$, and the right side will approach $\CH(b, \{D\}) \cup Q_{\{D\}}(a,c)$.

It remains to notice that $a,c$ cannot overlap, when $d=\{D\}$ and $abc\to d$ is tight. Indeed, by Lemma \ref{Overlap}, if $a,c$ overlap, then $D$ is in one of double hulls $\CH(b,c), \CH (a,b), \CH(a,c)$, thus, $abc\to d$ is not tight. 

Since $a,c$ do not overlap,  $Q_{\{D\}}(a,c)=\emptyset$. Therefore, $\CH(a,b,d)\cap \CH(b,c,d) = \CH(b,d)$.

Thus, it remains to prove the case when $d$ has a positive radius.

Note that it follows from the definition of $Q_d(ac)=\CH(a,d)\cap \CH(c,d)$ that $\CH(b,d)\cup Q_d(a,c)\subseteq \CH(a,b,d)\cap \CH(c,b,d)$. Moreover, set $\CH(a,b,d)\cap \CH(b,c,d)$ is convex; therefore, it will contain an area around every point not in $\CH(b,d)$ with points arbitrarily close to some portion of the border of $\CH(b,d)$. We will show that the only portion of the border of $\CH(b,d)$ for which $(\CH(a,b,d)\cap \CH(b,c,d))\setminus \CH(b,d)$ contains arbitrarily close points is the arc of attachment of $Q_d(a,c)$.

We would think of orientation of $a,b,c,d$ so that we would take advantage of circular order of touching points on circle $d$. Draw line $(BC)$ through centers of $b,c$, and then name touching points on $b$ as $B_1,B_2$ and point of intersection of $[DB)$ with $Circ(d)$ as $B_0$ so that moving from $B_1$ to $B_0$ to $B_2$ would be in counter clockwise order. By assumption, points $B_0,A_0,C_0$ appear in counter clockwise order. Besides, by Triangle Propoerty (Lemma \ref{triangle}), $D$ must be inside $\triangle ABC$, therefore, $A,C$ are separated by $(BD)$, and, thus, $C$ is in the same semiplane with respect to $(BD)$ as point $B_1$, and $B_2$ and $A$ are in the opposite semiplane. See Figure \ref{N} for illustration. 

By Theorem \ref{POrd}, we have the following point order on $d$: 
\[
B_1,C_1,A_1,B_2,C_2,A_2,B_1. 
\]

Note that $C_0$ is located on arc $\arck{C_2B_1C_1}$ because of the labelling assumption which states the counter clockwise points on $d$ from $\CH(c,d)$ is $\overset{\frown}{C_2 C_0 C_1}$. Thus, $B_2$ is contained on the complement arc $\overset{\frown}{C_1 C_{\infty} C_2}$. The intersection of $\arck{C_1 C_{\infty} C_2}$ and $\arck{B_2 B_{\infty} B_1}$ gives arc $\arck{B_2 C_2}$.

Furthermore, $A_0$ is located on arc $\arck{A_1B_2A_2}$, because of the labelling assumption, so $\arck{A_1B_2A_2} = \arck{A_1A_0A_2}$. Thus, $B_1$ is on arc $\arck{A_2 A_{\infty}A_1}$. By the labelling assumption and Theorem \ref{POrd}, arc $\arck{B_2 C_2 A_2 B_1} = \arck{B_2B_{\infty}B_1}$, so the intersection of $\arck{B_2B_{\infty}B_1}$ and $\arck{A_2A_{\infty}A_1}$ gives arc $\arck{A_2B_1}$.

Finally, $A_1$ is on arc $\arck{C_1 C_{\infty} C_2}$ and $C_1$ is on $\arck{A_2 A_{\infty} A_1}$, so $\arck{C_1 A_1}$ is a border arc of $\CH(a,c,d)$ on $d$. 

By Lemma \ref{c3-eq}, the following arcs on circle $d$ are border arcs of corresponding convex hulls:
\begin{itemize}
    \item   $\arck{B_2 C_2}$ is a border arc of $\CH(b,c,d)$; 
    \item $\arck{A_2 B_1}$ is a border arc of $\CH(a,b,d)$; 
    \item $\arck{C_1 A_1}$ is a border arc of $\CH(a,c,d)$.
\end{itemize}

 By Lemma \ref{c3-eq}, $\arck{C_2 A_2}$ is where $Q_d(ac)$ is attached because $\arck{C_1 A_1}$ is a border arc of $\CH(a,c,d)$. 
 
 Consider now the touching points on $b$. Denote by $D_1, D_2$ the touching points of the common tangents with $d$, so that $d,b$ are connected by common tangents $[B_1D_1]$ and $[B_2,D_2]$. Note that $D_1$ is in the same semiplane with respect to $(BD)$ as point $B_1$, and $D_2$ is in the opposite semiplane.

Denote by $C_b$ the touching point of the outer tangent between $b,c$, and $A_b$ the touching point of the outer tangent between $b,a$. 
Then moving around $b$ from $D_1$ to $D_2$, we will meet $C_b$, then $A_b$. Note that arc $\arck{C_bA_b}$ on $b$ is part of the border of $\CH(a,b,c)$, while arc $\arck{D_1C_bA_b}$ is part of the border of $\CH(a,b,d)$, and $\arck{C_bA_bD_2}$ is a part of the border of $\CH(c,b,d)$.

Here is the sequence of points on the border of $\CH(b,d)$, when walking around this binary hull:
\[
B_1,D_1,C_b,A_b,D_2,B_2,C_2,A_2,B_1,
\]
 and we summarize what we established earlier:
\begin{itemize}
    \item segment $[B_1D_1]$ is on the border of $\CH(a,b,d)$;
    \item arc $\arck{D_1C_bA_b}$ of $b$ is on the border of $\CH(a,b,d)$;
    \item arc $\arck{A_bD_2}$ of $b$ is on the border of $\CH(b,c,d)$;
    \item segment $[D_2B_2]$ is on the border of $\CH(b,c,d)$;
    \item arc $\arck{B_2C_2}$ on $d$ is a part of the border of $\CH(b,c,d)$;
    \item arc $\arck{C_2A_2}$ of $d$ is where $Q_d(a,c)$ is attached;
    \item arc $\arck{A_2B_1}$ of $d$ is a part of the border of $\CH(a,b,d)$.
\end{itemize}  

Therefore, there are no other points in $\CH(a,b,d)\cap \CH(b,c,d)$ than those in $\CH(b,d)\cup Q_d(a,c)$.
\end{proof}

Note. Since $Q_d(a,c)=\CH(a,d)\cap\CH(c,d)\setminus \tilde{d}$ and $\tilde{d}\subseteq \CH(b,d)$, we can rewrite the conclusion of Lemma \ref{Nested} as follows:
\[
\CH(a,b,d)\cap \CH(b,c,d) = \CH(b,d)\cup (\CH(a,d)\cap\CH(c,d))=
\]
\[
= (\CH(b,d)\cup \CH(a,d))\cap (\CH(b,d)\cup\CH(c,d))
\]
The left side of equality involves ``triple" convex hulls, while the right side is $\cup,\cap$-combination of ``double" convex hulls. Remarkably, the only assumption is the tight implication $abc\to d$.

\begin{theorem}\label{NT}
Assume that $abc\to d$ is a tight implication.
Suppose that $\tilde{e} \subseteq \CH(a,b,d) \cap \CH(c,b,d)$, but $\tilde{e}\not \subseteq \CH(b,d)$. Then $\tilde{e}\subseteq \CH(a,d)\cap \CH(c,d)$.
\end{theorem}
\begin{proof}
By assumption, we have $\tilde{e} \subseteq \CH(a,b,d)\cap \CH(c,b,d) = \CH(b,d)\cup Q_d(a,c)$. Since $\tilde{e} \not \subseteq \CH(b,d)$, we cannot have $r_d=0$ and $Q_d(a,c) = \emptyset$, thus, we must have point $E_x$ from $e$ in $Q_d(a,c)$. 

By Lemma \ref{Lemma42},
if circle $p$ is inscribed into angle, and $w_1,w_2$ are two areas: $w_1$ outside circle $p$ and next to vertex of angle, and another outside of circle $p$ and far from the vertex - then any other circle inside that angle cannot have points in both $w_1$ and $w_2$. The version of this Lemma is also true in the case the tangents forming $Q_d(a,c)$ are parallel.

\begin{figure}[H]
    \centering
        \begin{tikzpicture}[black, thick]
        \node[draw,circle,minimum size=10mm,outer sep=0] (b) {};
        \node[draw,circle,xshift=0.4cm,yshift=-1.7cm,minimum size=20mm,outer sep=0] (d) {};
        \node[draw,circle,xshift=2.5cm,yshift=-4.2cm,minimum size=22mm,outer sep=0] (c) {};
        \node[draw,circle,xshift=-2.5cm,yshift=-4.2cm,minimum size=30mm,outer sep=0] (a) {};
        
        \draw (tangent cs:node=b,point={(c.east)},solution=1) -- (tangent cs:node=c,point={(b.north)},solution=2);
        \draw (tangent cs:node=b,point={(d.east)}, solution=1) -- (tangent cs:node=d,point={(b.east)}, solution=2);
        \draw (tangent
        cs:node=b,point={(d.west)}, solution=2) -- (tangent cs:node=d,point={(b.west)}, solution=1);
       
        \draw (tangent cs:node=a,point={(c.south)},solution=2) -- (tangent cs:node=c,point={(a.south)});
        
        \draw (tangent cs:node=b,point={(a.west)},solution=2) -- (tangent cs:node=a,point={(b.west)});
        
        \draw (tangent cs:node=a,point={(d.east)},solution=2) -- (tangent cs:node=d,point={(a.south)});
         \draw (tangent cs:node=d,point={(c.south)},solution=2) -- (tangent cs:node=c,point={(d.west)});
        
        \draw[black] (-0.05, 0.75)node {$b$};
         \draw[black] (0.3, -1.7)node {$d$};
         \draw[black] (0.35, -2.9)node {$w_1$};
        \draw[black] (2.5, -4.15)node {$c$};
        \draw[black] (-2.5, -4.15)node {$a$};
          
        \end{tikzpicture}
    \caption{$w_1= Q_d(ac)$}
    \label{fig:proof19}
\end{figure}
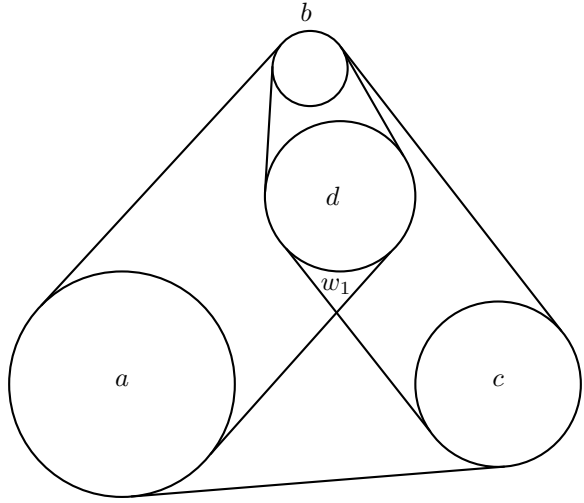

We have an application of this Lemma here: vertex of the angle is a point of intersection of tangents forming $Q_d(a,c)$, and circle $d$ plays the role of $p$. Then $Q_d(a,c)$ is contained either in $w_1$ or $w_2$ from that Lemma. Say, the case on Figure \ref{fig:proof19} is $Q_d(a,c)=w_1$. 
 
Thus, if $e$ has a point in $Q_d(a,c)$, and $e$ is contained in angle formed by extension of sides of $Q_d(a,c)$ from its vertex (which follows from $\tilde{e}\subseteq \CH(b,d)\cup Q_d(a,c)$), then $\tilde{e} \subseteq \tilde{d}\cup Q_d(a,c)$. But then $\tilde{e} \subseteq \CH(a,d)$, because $\tilde{d}\cup Q_d(a,c)\subseteq \CH(a,d)$. Similarly, $\tilde{e}\subseteq \CH(c,d)$.
\end{proof}

We will say that implication $X\to d$ on circles is loose, if $d \in ch_c(X)$.

\begin{corollary}\label{C54}
Let $a,b,c,d,e$ be circles on the plane. If $abc\to d$ is tight, $abd\to e$ and $cbd\to e$ are loose, then either $bd\to e$, or $ad \to e$ and  $cd\to e$.
\end{corollary}
\begin{proof}
Indeed, if $e \not \in  ch_c(b,d)$, then, by Theorem \ref{NT}, we will conclude that $e\in ch_c(a,d)$ and $e\in ch_c(c,d)$.
\end{proof}

 In the statement of the next corollary, we use the numbering of geometries given in [11] along with their implications. Elements of geometries 12, 15, 21, 23, 27, 33, and 47 have been relabeled to better fit the statement of Corollary \ref{C54}. 

\begin{corollary}\label{List2}
The following geometries, identified here by their tight implications, are not representable by circles on the plane: 

Geometry 12: $abc \to d$, $abd \to e$, and  $cbd \to e$. 

Geometry 15: $abc \to d$, $ab \to e$, and  $cbd \to e$. 

Geometry 18: $abc \to d$, $abd \to e$, $cbd \to e$, and $acd\to e$. 

Geometry 21: $abc \to d$, $abd \to e$, $cbd \to e$, and $ace\to d$.  

Geometry 23: $abc \to d$, $abd \to e$, $cbd \to e$, and $ac\to e$.  

Geometry 26: $abc \to d$, $ad \to e$, and  $cbd \to e$.   

Geometry 27: $abc \to d$, $ab \to e$, $cbd \to e$, and $ace\to d$.

Geometry 33: $abc \to d$, $ab \to e$, and $cb \to e$.

Geometry 35: $abc \to d$, $ad \to e$, $cbd \to e$, and $ab\to e$.

Geometry 45: $abc \to d$, $ad \to e$, and $cb \to e$.

Geometry 47: $abc \to d$, $ab \to e$, $cb \to e$, and $ace\to d$. 

Geometry 49: $abc \to d$, $ab \to e$, $cbd \to e$, and $ac\to e$.

Geometry 56: $abc \to d$, $ad \to e$, $cb \to e$, and $ab\to e$.

Geometry 60: $abc \to d$, $ad \to e$, $cbd \to e$, $ac\to e$, and $ab\to e$.

Geometry 70: $abc \to d$, $ab \to e$, $cb \to e$, and $ac\to e$.

Geometry 89: $abc \to d$, $a \to e$, and $cbd \to e$.

Geometry 94: $abc \to d$, $ad \to e$, $cb \to e$, $ac\to e$, and $ab\to e$.

Geometry 134: $abc \to d$, $a \to e$, and $bc \to e$.
\end{corollary}
\begin{proof}
In all these geometries the assumptions of Corollary \ref{C54} hold, but the conclusion either does not follow from given implications, or make one of them not tight.  
\end{proof}

\subsection{Surrounding and capturing}

The following result describes the situation of ``capturing" circle $e$ surrounded within a special set-up of 4 other circles.

\begin{theorem}\label{G74}
If $abc\to d$ is tight and $ad\to e$, $bd\to e$, $cd\to e$, then $d\to e$.
\end{theorem}
\begin{proof}
Suppose the statement is proved when $r_d>0$, and we have 5 circles with all assumptions of Theorem and $d=\{D\}$. Since $abc\to d$ is tight, $D$
 is an inner points of $\CH(a,b,c)$. Therefore, we may set a small positive number as $r_d$ so that $abc\to d$ still holds, as well as other assumptions $ad\to e$, $bd\to e$, $cd\to e$. Since we assumed that statement is proved for the case $r_d>0$, we conclude that $d\to e$. This will hold for arbitrary $r_d>0$, therefore, $e=\{D\}=d$, and we are done.
 Thus, we only prove Theorem under additional assumption that $r_d>0$.

By Theorem \ref{POrd}, using the same labeling, the counterclockwise point order on $d$ beginning at $B_1$ is  $B_1, C_1, A_1, B_2, C_2, A_2$. 

Now, by Lemma \ref{Nested} arc $\overset{\frown}{A_2C_1C_2}$ is part of the border for $\CH(a,d)\cap \CH(c,d)$, and arc $\overset{\frown}{C_1C_2B_1}$ is part of the border for $\CH(b,d)\cap \CH(c,d)$, see Figure \ref{N}.

 $A_2 \neq B_1$, else $d \in ch_c(a,b)$ and $abc \to d$ is not tight, and $C_1 \neq C_2$, since $d$ is not contained in $c$, so the complements of the arcs mentioned above are disjoint. Therefore, the two arcs forming the borders of the convex hulls include the entire border of $d$. It follows that there is no point outside $d$ from $\CH(a,d)\cap \CH(b,d)\cap \CH(c,d)$. Thus $Q_d(abc)=\emptyset$ and $d\to e$.
 \end{proof}

\begin{corollary}\label{List3}
The following geometries, identified here by their tight implications, are not representable by circles on the plane:

Geometry 74: $abc \to d$, $ad \to e$, $bd\to e$  and  $cd \to e$. 

Geometry 96: $abc \to d$, $ad \to e$, $bd\to e$, $cd \to e$ and $ab\to e$.

Geometry 105: $abc \to d$, $ad \to e$, $bd\to e$, $cd \to e$, $ab\to e$ and $ac\to e$. 

Geometry 143: $abc \to d$, $a \to e$, $bd\to e$  and  $cd \to e$.   

Geometry 147: $abc \to d$, $ad \to e$, $bd\to e$, $cd \to e$, $ab\to e$, $bc\to e$ and $ac\to e$.  

Geometry 206: $abc \to d$, $a \to e$, $bd\to e$,$cd \to e$ and $bc\to e$.  

Geometry 235: $abc \to d$, $a \to e$, $b\to e$  and  $cd \to e$.

Geometry 351: $abc \to d$, $a \to e$, $b\to e$  and  $c \to e$.
\end{corollary}
\begin{proof}
In all these geometries the assumptions of Theorem \ref{G74} hold, but the conclusion does not follow from given implications.  
\end{proof}

 \section{Representation of geometries by ellipses}\label{ellipse}

In paper \cite{Poly20} it was shown that from 672 known non-isomorphic geometries on 5-element set 623 can be represented by circles on the plane. From remaining 49, 8 are not representable due to the Weak carousel property from \cite{AdBo19} or the Triangle property from \cite{Poly20}. It turns out that all of these 49 geometry can be represented by ellipses. See appendix for the representations.

In recent paper \cite{Kin17}, it was shown that there exist convex geometries not representable by ellipses. The result provides 
Erd\"os-Szekeres type of obstruction and follows from specific construction in \cite{DHH16} of arbitrary large families of convex sets on the plane with the list of specific properties, which cannot be realized by families of ellipses. 

But no specific convex geometry defined by convex sets or implications was found so far that cannot be represented by ellipses. Also, the minimal cardinality of the base set of such convex geometry is not yet known.

\section{Geometries with colors}\label{predicates}

In this section, we propose a method of expanding the number of convex geometries which can be represented with circles in the plane by defining several unary predicates on the set of circles in representations. For visualization purposes unary predicates can be shown as colors.

\subsection{An operator on systems with unary predicates}

In this section we consider a general approach of defining a special operator on sets with unary predicates.

In general a $k$-ary \emph{predicate} on set $X$ is a subset of $X^k$. In particular, for $k=1$, a unary predicate $P$ is a subset of $X$. Predicates are parts of the language of general algebraic systems, that may have, besides predicates, some functions defined on base set of a system. A unary predicate can distinguish elements from $X$ with a special property.

Suppose $P_1,\dots, P_s$ are $s$ unary predicates on set $X$, thus, we think of $\langle X; \{P_1,\dots, P_s\}\rangle$ 
as an algebraic predicate system. For every $z\in X$ define $P(z)=\{k: z \in P_k\}$.

Define an operator $\mathcal{S}:2^X\rightarrow 2^X$ as follows:
\[
\mathcal{S}(Y)= \{ z \in X: P(z) \subseteq \bigcup_{y\in Y} P(y)\}.
\]

\begin{lemma}
$\mathcal{S}$ is a closure operator on $X$. 
\end{lemma}
\begin{proof}
Apparently, this operator is increasing and monotone. So we only need to check that it is idempotent: $\mathcal{S}(\mathcal{S}(Y))= \mathcal{S}(Y)$. Indeed, $\mathcal{S}(\mathcal{S}(Y))=\{z\in X: P(z) \subseteq \bigcup_{y\in \mathcal{S}(Y)} P(y)\}= \{z\in X: P(z) \subseteq \bigcup_{y\in Y} P(y)\}= \mathcal{S}(Y)$. 
\end{proof}

The following example shows that $\mathcal{S}$ does not necessarily satisfy the anti-exchange property.  

\begin{example}\label{OpS}
\end{example}
Consider $X=\{x,z,y_1,y_2, w\}$ and $P_1=\{x,z,y_2\}$, $P_2=\{z,y_1\}$. Then $P(x)=\{1\}, P(z)=\{1,2\},P(y_1)=\{2\}, P(y_2)=\{1\}, P(w)=\emptyset.$ One can check that the range of $\mathcal{S}$ in $2^X$ is $\{w, y_1w, xy_2w, X\}$.

Note that $\mathcal{S}$ does not satisfy the anti-exchange property:
set $xy_2w$ covers $w$ in the lattice of closed sets, but the difference between the two subsets is 2 elements.\hspace{3.5 cm} \qedsymbol{}

\vspace{0.5cm}

Suppose now that we have another closure operator $\phi:2^X\rightarrow 2^X$ on $X$. We then can define an associated operator $\phi_\mathcal{S}:2^X\to 2^X$ 
on $\langle X; \{P_1,\dots, P_s\}\rangle$ as follows:
\[
\phi_\mathcal{S}(Y)=\phi(Y)\cap \mathcal{S}(Y).
\]

\begin{lemma}
$\phi_\mathcal{S}$ is a closure operator. Moreover, $\phi_\mathcal{S}(Y)\subseteq \phi(Y)$, for any $Y\in 2^X$, therefore, all $\phi$-closed sets are simultaneously $\phi_\mathcal{S}$-closed. 
\end{lemma}
\begin{proof}
Since both $\phi$ and $\mathcal{S}$
are closure operators, so is operator $\phi_\mathcal{S}$, which incidentally equals the greatest lower bound of two operators on poset of closure operators on $X$.
In particular, $\phi_\mathcal{S}(Y)\subseteq \phi(Y)$ for any $Y\in 2^X$. Therefore, every $\phi(Y)$ is also closed for $\phi_\mathcal{S}$. Indeed, $\phi(Y)\subseteq \phi_\mathcal{S}(\phi(Y))\subseteq \phi(\phi(Y))=\phi(Y)$, i.e., $\phi_\mathcal{S}(\phi(Y))= \phi(Y)$.  
\end{proof}

The poset of closure operators on set $X$ was studied in J.B. Nation
\cite{Nat04}.

Unfortunately, operator $\phi_\mathcal{S}$ might not satisfy the anti-exchange property, even if $\phi$ does. This is due to the earlier observation that $\mathcal{S}$ may fail the anti-exchange property.
\begin{example}\label{NoAEP}
\end{example}

Consider 4-element base set $X=\{x,z, y_1,y_2\}$ and linear convex geometry $(X,\phi)$ with following convex sets:
$\{\emptyset, x, xy_1, xzy_1, xzy_1y_2\}$. 
Suppose $P_1=\{x,z\}$ and $P_2=\{y_1,y_2\}$ are two predicates on $X$.

Then $\phi_\mathcal{S}(\{y_1,y_2\}) = \{y_1,y_2\}$, because $\phi$-closure is the whole set and $\mathcal{S}(\{y_1,y_2\})\\=\{y_1,y_2\}$. Thus, $Y=\{y_1,y_2\}$ is $\phi_\mathcal{S}$-closed set. Moreover, $z$ is in $\phi_\mathcal{S}(Y \cup \{x\})$, because $z$ is in $P_1$ together with $x$, and $z$ is in $\phi(Y \cup \{x\})$. But the same is true when they are switched: $x$ is in $\phi_\mathcal{S}(Y \cup \{z\})$. Thus, we can exchange elements $x,z$ and anti-exchange fails. This is shown by the fact that in the lattice of closed sets of $\phi_\mathcal{S}$ from element $Y$ to $X$ there is no progression by adding one element: one would jump by two elements $x,z$.\hspace{4.5cm}\qedsymbol{}
\vspace{0.5cm}

\begin{example}\label{WithAEP}
\end{example}
We could change the predicates on $X$ from Example \ref{NoAEP}  and get an extension of convex geometry $(X,\phi)$ in that example. Let $P_1=\{z,y_2\}$, $P_2=\{x\}$. Then $P(y_1)=\emptyset, P(z)=P(y_2)=\{1\}$, $P(x)=\{2\}$. 
 
We verify that there are new closed sets $\{y_1\}, \{zy_1\},\{z,y_1,y_2\}$ for operator $\phi_\mathcal{S}$. Moreover, the operator is with the anti-exchange property.\hspace{2cm}\qedsymbol{}

\vspace{0.5cm}

What was shown in Example \ref{WithAEP} is possible to achieve for convex geometries given by circle configuration. Introducing one, two or three predicates to circle configuration, it is possible to represent all 5-element geometries that are not representable by circles only.

\begin{example}\label{G134with color}
\end{example}

It was shown in Corollary \ref{List2} that G134 is not representable by circles. G134 is given by the following implications on $X=\{a,b,c,d,e\}$:
\[
abc\to d, bc\to e, a\to e
\]

The geometry $G=(X,\CH)$ on same set $X$ as circles, with closure operator $\CH$, represented on Figure \ref{fig:G134} satisfies these implications, except $abc\to d$ is not tight, and can be replaced by $bc\to d$. 

\begin{figure}[H]
    \centering
        \begin{tikzpicture}
            \draw (2.16431605453767,1.11889327842919) [color = {rgb,255:red,153;green,51;blue,255}] circle (1.820495447927309cm);
            \draw (1.0268019607180707,-0.5053555499986746) [color = {rgb,255:red,0;green,0;blue,255}] circle (1.1713265285113694cm);
            \draw (3.684148903695332,-0.5138464250981529) [color = {rgb,255:red,0;green,0;blue,255}] circle (1.1713265285113694cm);
            \draw (2.673471597346966,-1.5264597597974565) [color = {rgb,255:red,204;green,0;blue,0}] circle (0.04233710344016998cm);
            \draw (2.197982591776191,-0.17641061898039612) [color = {rgb,255:red,0;green,0;blue,255}] circle (0.26813498845440986cm);
            \draw[gray] (2.1045829656819315,1.2245837724335344)node {$a$};
            \draw[gray] (1.0177509529487316,-0.43113687196474704)node {$b$};
            \draw[gray] (3.6499222337869504,-0.3886824964673552)node {$c$};
            \draw[gray] (2.5140354716506193,-1.424569258603716)node {$d$};
            \draw[gray] (2.1810008415772346,-0.05753836758769891)node {$e$};
        \end{tikzpicture}
    \caption{$abc\to de, bc\to e, a\to e, C_1=\{a,b,c,e\},C_2=\{a,d\}$}
    \label{fig:G134}
\end{figure}

It can be checked directly that the family of closed sets of G134 has $bce$, which is not closed in $G$.

We can introduce two predicates $C_1=\{a,b,c,e\}$ and $C_2=\{a,d\}$ on $X$, which results in the following family of closed sets for operator $\mathcal{S}$:
\[
\emptyset, d, bce, X
\]

Note that sets $\emptyset$ and $d$ are already closed in $G$, and proper intersections of closed sets of $G$ with $bce$ give $b,c,e,be,ce$, which are already $\CH$-closed. Therefore, $bce$ is the only new closed set of operator $\CH_\mathcal{P}$ compared to $\CH$. Moreover, set $bcde$ is $\CH$-closed set, which is one-point extension of $bce$. This means that $\CH_\mathcal{P}$ has exactly one more closed set than $\CH$ and it satisfies the anti-exchange property.\hspace{10.5cm}\qedsymbol{}

\vspace{0.5cm}

\subsection{Sufficient condition for representation with unary predicates}

We are going to formulate a sufficient property of operator $\mathcal{S}$ so that a closure system with operator $\phi$ can be expanded by $\mathcal{S}$-closed sets so that $\phi_\mathcal{S}$ still satisfies the anti-exchange property. Let $\mathcal{C}_\phi$ be the family of $\phi$-closed sets.

Recall, say, from \cite{CaMo03} that subset $Y\subseteq X$ is called \emph{quasi-closed} w.r.to operator $\phi$, if it is not in $\mathcal{C}_\phi$ and $\mathcal{C}_\phi \cup \{Y\}$ is closed under the intersection, i.e., it is a family of closed sets of another closure operator. Such property is equivalent to:
\[
Z\cap Y \in \mathcal{C}_\phi \text { or } Z\cap Y = Y
\]
for all $Z\in \mathcal{C}_\phi$.

\begin{lemma}\label{suffAEP}
Suppose $\phi$ is a closure operator with the anti-exchange property defined on set $X$, and operator $\mathcal{S}$ is defined by unary predicates $P_1,\dots, P_k$ on $X$. Assume that every $\mathcal{S}$-closed set is ether $\phi$-closed or quasi-closed w.r.t. $\phi$, and in latter case has one-element extension in $\mathcal{C}_\phi$. Then $\phi_\mathcal{S}$ satisfies the anti-exchange property as well.
\end{lemma}
\begin{proof}
$\phi_\mathcal{S}$-closed sets are of the form $C\cap L$, where $C \in \mathcal{C}_\phi$ and $L$ is $\mathcal{S}$-closed. By the definition, it is either in $\mathcal{C}_\phi$, or it is $L$. Thus only $\mathcal{S}$-closed sets may be added to the family of $\phi$-closed sets. By assumption, all $\mathcal{S}$-closed sets have one-element extensions in $\mathcal{C}_\phi$, thus, adding them to $\mathcal{C}_\phi$ will give us a closure system with one-point extension, which is equivalent to $\phi_\mathcal{S}$ to satisfy the anti-exchange property.
\end{proof}

Revisiting Example \ref{G134with color}, note that the closed sets of operator $\mathcal{S}$ are
\[
\emptyset, d, bce, X
\]
From these, only $bce$ is not $\phi$-closed. We check that all proper intersections of this set with $\phi$-closed sets give $c,e,be,b, ce$, which are $\phi$-closed. Therefore, $bce$ is quasi-closed. It also has one-element extension among $\phi$-closed sets. Therefore, Lemma \ref{suffAEP} applies.

The next example shows that Lemma \ref{suffAEP} allows some further generalization.

\begin{example}\label{G7}
\end{example}
Consider representation of $G7$ (id: 4294924159 in \cite{Poly20}) by circles with predicates given by implications
\[
acd\to e, abd\to e, abc\to e
\]

We start with representation of $G7$ on 4-element set $\{b,c,d,e\}$, with id:65303 from \cite{Poly20}. It is defined by implications
\[
cd\to e, bd\to e, bc\to e
\]

\begin{figure}[htp]
    \centering
    \includegraphics[width=5cm]{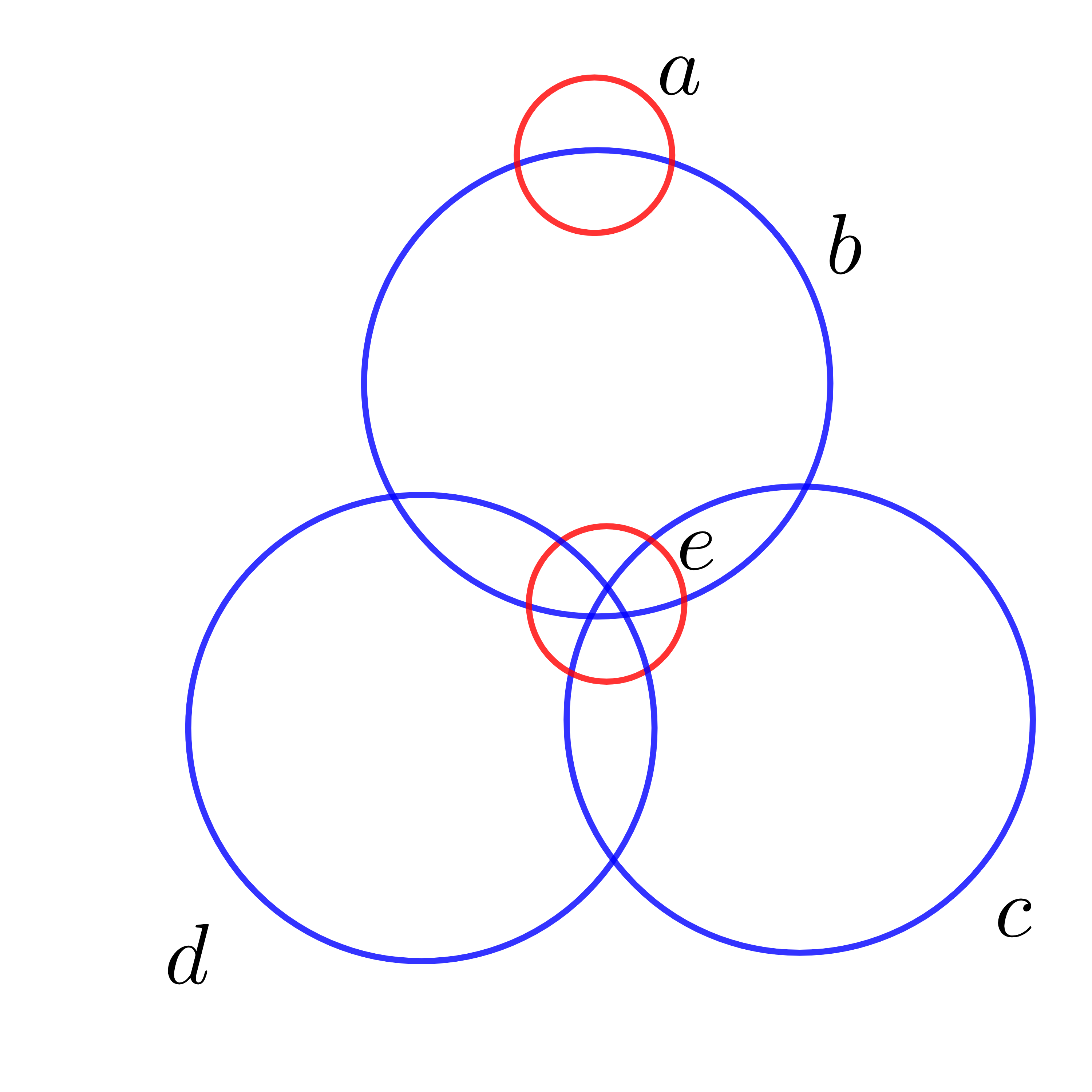}
    \caption{}
    \label{fig:G7}
\end{figure}

If we ignore color-predicates on Figure \ref{fig:G7}, then we get representation of G52(id: 4294907671 in \cite{Poly20}) on $X=\{a,b,c,d,e\}$ with implications 
\[
cd\to e, bd\to e, bc\to e,
\]
i.e., $b,c,d,e$ represent $G7$ on $\{b,c,d,e\}$, and $a$ is added so that $a$ does not participate in any implication.

We can introduce two predicates on $X$: red and blue, where $R=\{a,e\}$ and $B=\{b,c,d\}$, as on Figure \ref{fig:G7}. This gives the $\mathcal{S}$-closed sets:
\[
\emptyset, ae,bcd,X
\]
This time, $\emptyset, ae,X$ are closed in $G52=(X,\phi)$, while $bcd$ is neither closed nor quasi-closed, and intersections of $bcd$ with $\phi$-closed sets give non-closed sets for $\phi$: 
\[
bc,bd,cd.
\]
All of these are quasi-closed w.r.t. $\phi$, and they have one-point extensions in $C_\phi$: $bce, bde,cde$, respectively. Similarly, $bcd$ has one-point extension $bcde$ in $C_\phi$. Thus, sets $bc,cd,bd,bcd$ can be added to $C_\phi$ preserving one-point extension property, which guarantees that $C_\phi$ has AEP.

Adding $bc,cd,bd,bcd$ to $\phi$-closed sets, make all subsets of $\{b,c,d,e\}$ $\phi_\mathcal{S}$-closed, but $abc,acd,abd,abcd$ are still not closed, which gives needed implications describing $\phi_\mathcal{S}$:
\[
acd\to e, abd\to e, abc\to e.
\]

This example shows that we can extend assumptions of Lemma \ref{suffAEP} to sets that are the \emph{intersections} of $\phi$-closed and $\mathcal{S}$-closed sets, and obtain the same conclusion.

\subsection{Geometry that requires more than 2 colors}

In Appendix A we show colored representations of all convex geometries on 5-element set which cannot be represented by circles.
Most of them require only one or two colors for representation by colored circles.
 Only one geometry, G18 requires three colors, which is demonstrated in this section.

\begin{proposition}
Convex geometry G18 on 5-element set defined by implications: 
\[abc\to de,acd\to e,bcd\to e,abd\to e\]
cannot be represented by circles with coloring by at most two colors.
\end{proposition}
\begin{proof}
 Assume that we have representation of G18 by 5 circles with at most two colors $c_1$ and $c_2$. Let $A,B,C,D,E$ be color sets of $a,b,c,d,e$ respectively.  
 Since we have $abc\to d$ tight, $abd\to e$ and $bcd\to e$, by Corollary \ref{C54} we have either $bd\to e$ or $ad\to e$ and $cd\to e$ that hold by the convex hall operator on circles. We study both of these cases and show that we get to a contradiction.

 \textbf{Case 1:} Assume that $ad\to e$ and $cd\to e$ hold.  Since $ad\to e$ and $cd\to e$ by the convex hull operator, we need to have some color $c_1 \in E$ such that $c_1 \notin D$, $c_1 \notin C$ and $c_1 \notin A$. Thus $A,C,D \subseteq  \{c_2\}$. Thus $acd\to e$ does not hold with the new $\CH_C$ operator, since $c_1 \in E$ and $A,C,D = \{c_2\}$. This brings to the contradiction. \\
 
 \textbf{Case 2:} $bd \to e$. We further split it into two sub-cases: \\
    
 \textbf{\textit{Case 2.1:}} Suppose $cd\to e$ holds for the convex hall operator on circles. Since $bd\to e$ and $cd\to e$ by the convex hull operator, we need to have some color $c_1 \in E$ such that $c_1 \notin D$, $c_1 \notin C$ and $c_1 \notin B$. Thus $B,C,D \subseteq  \{c_2\}$ and $bcd\to e$ does not hold with the new $\CH_C$ operator. This brings to the contradiction. \\
 
 \textbf{\textit{Case 2.2:}} Assume that $cd\not\to e$. We use Corollary \ref{C54} again, relabelling b with c. Since $abc\to d$ tight, $acd\to e$ and $bcd\to e$, by Corollary \ref{C54} we have either $cd\to e$ or $ad\to e$ and $bd\to e$. Now, since $cd\not\to e$, we have $ad\to e$ and $bd\to e$. Since $ad\to e$ and $bd\to e$ hold for the convex hull operator, we need to have some color $c_1 \in E$ such that $c_1 \notin D$, $c_1 \notin B$ and $c_1 \notin A$. Thus $A,B,D \subseteq \{c_2\}$. Thus, $abd\to e$ does not hold with the new $\CH_C$ operator, bringing to the contradiction.
 \end{proof}

On the other hand, the representation with three colors is possible and it is shown on Figure \ref{fig:G18}.

\section{Acknowledgments.}
The paper follows up \cite{Poly20} completed by the team of 18 undergraduate and 2 graduate students in 2020, in the framework of PolyMath - 2020 project. We thank participants of the original project who continued to support the work of the current group in various capacities:  Fernanda Yepez-Lopez, Rohit Pai and Stephanie Zhou.

\textit{Address for correspondence:}\\
Kira Adaricheva, Department of Mathematics, Hofstra University, Hempstead NY, 11549\\
\textit{Email address:} \href{mailto:kira.adaricheva@hofstra.edu}{\texttt{Kira.Adaricheva@Hofstra.edu}}

\newpage

\appendix
\appendixpage
\addappheadtotoc
\section{Ellipse Representations}
\include{
anc/AppendixA}
\section{Circle Coloring Representations}
\include{
anc/AppendixB}



\newpage

 \begin{thebibliography}{99}

\bibitem{Lemons21} K. Adaricheva, B. Brubaker, P. Devlin, S. J. Miller, V. Reiner, A. Seceleanu, A. Sheffer, and Y. Zeytuncu, \emph{When life gives you lemons, make mathematicians}, Notices of AMS, March 2021, 375--78. 

\bibitem{AdBo19} K. Adaricheva, M. Bolat, \emph{Representation of finite convex geometries by circles on the plane}, Discrete Mathematics v.342 (2019), N3, 726-746. 


\bibitem{AGT03} Adaricheva K.V., Gorbunov V.A., Tumanov V.I., \emph{Join-semidistributive lattices and convex geometries}, Advances in Mathematics 173(2003), 1-49.

\bibitem{AdNa16} K. Adaricheva, J.B.Nation, \emph{Convex geometries}, in Lattice Theory: Special Topics and Applications, v.2, G.~Gr\"atzer and F.~Wehrung, eds. Springer, Basel, 2016.

\bibitem{AN16I} K. Adaricheva, J.B.Nation, \emph{Bases in closure systems}, in Lattice Theory: Special Topics and Applications, v.2, G.~Gr\"atzer and F.~Wehrung, eds. Springer, Basel, 2016.


\bibitem{CaMo03} N.~Caspard and B.~Monjardet, \emph{The lattices of closure systems, closure operators, and implicational systems on a finite set: a survey}, Disc. Appl. Math. \textbf{127} (2003), 241--269.

\bibitem{Cz14} 
G.~Cz\'edli,
\textit{Finite Convex Geometries of Circles}.  
Discrete Mathematics 330 (2014), 61--75.

\bibitem{DHH16} M.G.Dobbins, A.Holmes, A.Hubard, \emph{Regular systems of paths and families of convex sets in convex position}, Trans. AMS 368 (2016), N5, 3271--3303.


\bibitem{EdJa85}  
P.H.~Edelman and R.E.~Jamison,
\textit{The Theory of Convex Geometries}.  
Geom Dedicata 19 (1985), 247--274.

\bibitem{EdSa88} P.H.~Edelman and M.~Saks, \textit{Combinatorial representation and convex dimension of convex geometries}. Order 5 (1988), 23--32.

\bibitem{ggb}
GeoGebra Toolkit \url{https://www.geogebra.org/classic/dk88mrgu}

\bibitem{Kin17} J. Kincses, \emph{On the representation of finite convex geometries with convex sets}, Acta Sci. Math.83 (2017).

\bibitem{Nat04} J.B.Nation, \textit{Closure operators and lattice extensions}. Order 21 (2004), 43--48.

\bibitem{Poly20} PolyMath REU-2020, \emph{Convex geometries representable by at most 5 circles on the plane}, arxiv: 2008.13077


\end{thebibliography}
\end{document}